\documentclass{amsart}

\usepackage{enumerate,amssymb,amsmath,mathrsfs}
\usepackage{amssymb,amsfonts,amsthm,amscd}
\usepackage[mathscr]{eucal}     
\usepackage{graphicx}
\usepackage[arrow, matrix, curve]{xy}
\usepackage{color}
\definecolor{gray}{gray}{.75}
\definecolor{gray2}{gray}{.50}

\newcommand{\one}{\mathbf{1}}
\newcommand{\dom}{\mathcal{D}}
\newcommand{\domr}{\mathcal{D}_{\min}}

\newcommand{\doma}{\mathcal{D}_{\max}}

\newcommand{\domm}{\widetilde{\mathcal{D}}}

\newcommand{\B}{\mathcal{B}}

\newcommand{\Z}{\mathbb{Z}}

\newcommand{\R}{\mathbb{R}}
\newcommand{\C}{\mathbb{C}}

\newcommand{\D}{\nabla}
\newcommand{\Dr}{\nabla_{\min}}
\newcommand{\Da}{\nabla_{\max}}
\newcommand{\DD}{\widetilde{\nabla}}
\newcommand{\G}{\Gamma}
\newcommand{\GG}{\widetilde{\Gamma}}

\newcommand{\A}{\mathrm{\alpha}}
\newcommand{\w}{\mathrm{\omega}}

\newtheorem{thm}{Theorem}[section]
\newtheorem{cor}[thm]{Corollary}
\newtheorem{remark}[thm]{Remark}
\newtheorem{prop}[thm]{Proposition}
\newtheorem{lemma}[thm]{Lemma}
\newtheorem{defn}[thm]{Definition}

\numberwithin{equation}{section}

\begin{document}

\title[Refined Analytic Torsion]
{Refined Analytic Torsion on Manifolds with Boundary}
\author{Boris Vertman}
\address{University of Bonn \\ Department of Mathematics \\
Beringstr. 6\\ 53115 Bonn\\ Germany}
\email{vertman@math.uni-bonn.de}

\thanks{2000 Mathematics Subject Classification. 58J52.}

\begin{abstract}
{We discuss the refined analytic torsion, introduced by M. Braverman and T. Kappeler as a canonical refinement of analytic torsion on closed manifolds. Unfortunately there seems to be no canonical way to extend their construction to compact manifolds with boundary. We propose a different refinement of analytic torsion, similar to Braverman and Kappeler, which does apply to compact manifolds with and without boundary. In a subsequent publication we prove a surgery formula for our construction.}
\end{abstract}

\maketitle

\pagestyle{myheadings}
\markboth{\textsc{Refined Analytic Torsion}}{\textsc{Boris Vertman}} 

\section{Introduction}\
\\[-3mm] The refined analytic torsion, defined by M. Braverman and T. Kappeler in [BK1] and [BK2] on closed manifolds, can be viewed as a refinement of the Ray-Singer torsion, since it is a canonical choice of an element with Ray-Singer norm one, in case of unitary representations. 
\\[3mm] The complex phase of the refinement is given by the rho-invariant of the odd-signature operator. Hence one can expect the refined analytic torsion to give more geometric information than the Ray-Singer torsion.
\\[3mm] This is indeed the case in the setup of lens spaces with explicit formulas for the associated Ray-Singer torsion and eta-invariants, see [RH, Section 5] and the references therein. There it is easy to find explicit examples of lens spaces which are not distinguished by the Ray-Singer torsion, however have different rho-invariants of the associated odd-signature operators.
\\[3mm] An important property of the Ray-Singer torsion norm is its gluing property, as established by W. L\"{u}ck in [L\"{u}] and S. Vishik in [V]. It is natural to expect a refinement of the Ray-Singer torsion to admit an analogous gluing property. 
\\[3mm] Unfortunately there seems to be no canonical way to extend the construction of Braverman and Kappeler to compact manifolds with boundary. In particular a gluing formula seems to be out of reach.
\\[3mm] We propose a different refinement of analytic torsion, similar to Braverman and Kappeler, which does apply to compact manifolds with and without boundary. In the subsequent publication [BV4] we establish a gluing formula for our construction, which in fact can also be viewed as a gluing law for the original definition of refined analytic torsion by Braverman and Kappeler.
\\[3mm] The presented construction is analogous to the definition in [BK1] and [BK2], but applies to any smooth compact Riemannian manifold, with or without boundary. For closed manifolds the construction differs from the original definition in [BK2]. Nevertheless we still refer to our concept as "refined analytic torsion" within the present discussion. 
\\[3mm] {\bf Acknowledgements.} The results of this article were obtained during the author's Ph.D. studies at Bonn University, Germany. The author would like to thank his thesis advisor Prof. Matthias Lesch for his support and useful discussions. The author was supported by the German Research Foundation as a scholar of the Graduiertenkolleg 1269 "Global Structures in Geometry and Analysis".

\section{Motivation for the generalized construction}\
\\[-3mm] The essential ingredient in the definition of the refined analytic torsion in [BK2] is the twisted de Rham complex with a chirality operator and the elliptic odd-signature operator associated to the complex, viewed as a map between the even forms. Hence in the case of a manifold with boundary we are left with the task of finding elliptic boundary conditions for the odd-signature operator which preserve the complex structure and provide a Fredholm complex, in the sense of [BL1].
\\[3mm] The notions of a Hilbert and a Fredholm complex were studied systematically in [BL1] and will be provided for convenience in the forthcoming section. The boundary conditions, that give rise to a Hilbert complex are referred to as "ideal boundary conditions". It is important to note that the most common self-adjoint extensions of the odd-signature operator between the even forms do not come from ideal boundary conditions. 
\\[3mm] The existence and explicit determination of elliptic boundary conditions for the odd-signature operator between the even forms, arising from ideal boundary conditions, is an open question. However, it is clear that the absolute and relative boundary conditions do not satisfy these requirements.  
\\[3mm] On the other hand the gluing formula in [V] and [L\"{u}] for the Ray-Singer torsion makes essential use of the relative and absolute boundary conditions. Since the establishment of a corresponding gluing formula for the refined analytic torsion is a motivation for our discussion, these boundary conditions seem to be natural choices.
\\[3mm] We are left with a dilemma, since neither the relative nor the absolute boundary conditions are invariant under the Hodge operator. We resolve this dilemma by combining the relative and absolute boundary conditions. This allows us to apply the concepts of [BK2] in a new setting and to establish the desired gluing formula. 

\section{Definition of Refined analytic torsion}\label{explicit-unitary}\
\\[-3mm] Let $(M^m, g^M)$ be a smooth compact connected odd-dimensional oriented Riemannian manifold with boundary $\partial M$, which may be empty. Let $(E, \D, h^E)$ be a flat complex vector bundle with any fixed Hermitian metric $h^E$, which need not to be flat with respect to $\D$.  
\\[3mm] The flat covariant derivative $\D$ is a first order differential operator $$\D: \Gamma (E) \rightarrow \Gamma (T^*M\otimes E),$$ satisfying the Leibniz rule $$\D_X(fs)=(Xf)s+f\D_Xs, \quad s \in \Gamma (E), X \in \Gamma (TM), f \in C^{\infty}(M).$$ The covariant derivative $\D$ extends by the Leibniz rule to the twisted exterior differential $\D: \Omega^k_0(M, E)\to \Omega^{k+1}_0(M, E)$ on $E-$valued differential forms with compact support in the interior of the manifold $\Omega^k_0(M,E)$. The exterior differential satisfies the (generalized) Leibniz rule 
\begin{align*}
\D_X(w\wedge \eta)=(\D_X w)\wedge \eta +(-1)^pw\wedge \D_X\eta,
\end{align*}
for any $w \in \Omega^p_0(M), \eta \in \Omega^q_0(M,E), X\in \Gamma (TM)$. 
\\[3mm] Due to flatness of $(E,\D)$ the twisted exterior differential gives rise to the twisted de Rham complex $(\Omega^*_0(M,E), \D)$. The metrics $g^M, h^E$ induce an $L^2-$inner product on $\Omega^*_0(M,E)$. We denote the $L^2-$completion of $\Omega^*_0(M,E)$ by $L^2_*(M,E)$. 
\\[3mm] Next we introduce the notion of the dual covariant derivative $\D'$. It is defined by requiring:
\begin{align}\label{dual-connection}
dh^E(u,v)[X]=h^E(\D_Xu,v)+h^E(u,\D'_Xv),
\end{align}
to hold for all $u,v \in C^{\infty}(M,E)$ and $X \in \Gamma(TM)$.
In the special case that the Hermitian metric $h^E$ is flat with respect to $\D$, the dual $\D'$ and the original covariant derivative $\D$ coincide. More precisely the Hermitian metric $h^E$ can be viewed as a section of $E^*\otimes E^*$. The covariant derivative $\D$ on $E$ gives rise to a covariant derivative on the tensor bundle $E^*\otimes E^*$, also denoted by $\D$ by a minor abuse of notation. 
\\[3mm] For $u,v,X$ as above one has:
$$\D h^E(u,v)[X]=dh^E(u,v)[X]-h^E(\D_{X}u,v)-h^E(u,\D_{X}v).$$
In view of \eqref{dual-connection} we find $$\D h^E=0 \ \Leftrightarrow \D=\D'.$$
As before, the dual $\D'$ gives rise to a twisted de Rham complex. Consider the differential operators $\D, \D'$ and their formal adjoint differential operators $\D^t, \D'^t$. The associated minimal closed extensions $\Dr, \Dr'$ and $\Dr^t, \Dr'^t$ are defined as the graph-closures in $L^2_*(M,E)$ of the respective differential operators. The maximal closed extensions are defined by
 $$\Da:=(\D^t_{\min})^*, \quad \Da':=(\D'^t_{\min})^*.$$
These extensions define Hilbert complexes in the following sense, as introduced in [BL1].
\begin{defn} \textup{[BL1]}
Let the Hilbert spaces $H_i,i=0,..,m,H_{m+1}=\{0\}$ be mutually orthogonal. For each $i=0,..,m$ let $D_i\in C(H_i,H_{i+1})$ be a closed operator with domain $\dom (D_i)$ dense in $H_i$ and range in $H_{i+1}$. Put $\dom_i:=\dom(D_i)$ and $R_i:=D_i(\dom_i)$ and assume $$R_i\subseteq \dom_{i+1}, \quad D_{i+1}\circ D_i=0.$$
This defines a complex $(\dom, D)$ $$0 \rightarrow \dom_0 \xrightarrow{D_0}\dom_1\xrightarrow{D_1}\cdots \xrightarrow{D_{m-1}}\dom_m\rightarrow 0.$$
Such a complex is called a Hilbert complex. If the homology of the complex is finite, i.e. if $R_i$ is closed and $\ker D_i / \textup{im} D_{i-1}$ is finite-dimensional for all $i=0,...,m$, the complex is referred to as a Fredholm complex.
\end{defn}\ \\
\\[-7mm] Indeed, by [BL1, Lemma 3.1] the extensions define Hilbert complexes as follows
\begin{align*}&(\domr, \Dr), \ \textup{where} \ \domr:=\dom (\Dr), \\  &(\doma, \Da), \ \textup{where} \ \doma:=\dom (\Da) \\[3mm]
&\hspace{20mm} (\domr', \Dr'), \ \textup{where} \ \domr':=\dom (\Dr'), \\  &\hspace{20mm} (\doma', \Da'), \ \textup{where} \ \doma':=\dom (\Da') . 
\end{align*} 
Note the following well-known central result on these complexes.
\begin{thm}\label{thm41} The Hilbert complexes $(\domr, \Dr)$ and $(\doma, \Da)$ are Fredholm with the associated Laplacians $\triangle_{\textup{rel}}$ and $\triangle_{\textup{abs}}$ being strongly elliptic in the sense of [Gi]. The de Rham isomorphism identifies the homology of the complexes with the relative and absolute cohomology with coefficients:
\begin{align*}
H^*(\domr, \Dr)&\cong H^*(M,\partial M, E), \\  H^*(\doma, \Da)&\cong H^*(M,E).
\end{align*}
Furthermore the cohomology of the Fredholm complexes $(\domr, \Dr)$ and $(\doma, \Da)$ can be computed from the following smooth subcomplexes, 
\begin{align*}
(\Omega^*_{\min}(M,E), \D), \quad &\Omega_{\min}^*(M,E):=\{\w \in \Omega^*(M,E)|\iota^*(\w)=0\}, \\
(\Omega^*_{\max}(M,E), \D), \quad &\Omega_{\max}^*(M,E):=\Omega^*(M,E),
\end{align*}
respectively, where we denote by $\iota: \partial M \hookrightarrow M$ the natural inclusion of the boundary.
\end{thm} \ \\
[-5mm] In the untwisted setup this theorem is essentially the statement of [BL1, Theorem 4.1]. The theorem remains true in the general setup. An analogue of the trace theorem [P, Theorem 1.9], in case of flat vector bundles, allows an explicit computation of the boundary conditions for $\triangle_{\textup{rel}}$ and $\triangle_{\textup{abs}}$. Then [Gi, Lemma 1.11.1] implies strong ellipticity of the Laplacians. Note that this result in the reference [Gi] is proved explicitly, even though other aspects of [Gi, Section 1.11] are rather expository.
\\[3mm] By strong ellipticity the Laplacians $\triangle_{\textup{rel}}$ and $\triangle_{\textup{abs}}$ are Fredholm and by [BL1, Theorem 2.4] the complexes $(\domr, \Dr)$ and $(\doma, \Da)$ are Fredholm as well. By [BL1, Theorem 3.5] their cohomology indeed can be computed from the smooth subcomplexes $(\Omega^*_{\min}(M,E), \D)$ and $(\Omega^*_{\max}(M,E), \D)$, respectively.
\\[3mm] Finally, the relation to the relative and absolute cohomolgy (the twisted de Rham theorem) is proved in [RS, Section 4] for flat Hermitian metrics, but an analogous proof works in the general case. Corresponding results hold also for the complexes associated to the dual connection $\D'$. 
\\[3mm] Furthermore, the Riemannian metric $g^M$ and the fixed orientation on $M$ give rise to the Hodge-star operator for any $k=0,..,m=\dim M$: $$*:\Omega^k(M,E)\to \Omega^{m-k}(M,E).$$ Define $$\G :=i^r(-1)^{\frac{k(k+1)}{2}}*:\Omega^k(M,E)\to \Omega^{m-k}(M,E), \quad r:= (\dim M+1)/2.$$ This operator extends to a well-defined self-adjoint involution on $L^2_*(M,E)$, which we also denote by $\G$. The following properties of $\G$ are essential for the later construction. 
\begin{lemma}\label{G-Lemma}
The self-adjoint involution $\G$ relates the minimal and maximal closed extensions of $\D$ and $\D'$ as follows $$\G\Dr \G=(\Da')^*, \quad \G\Da \G=(\Dr')^*.$$
\end{lemma}
\begin{proof}
One first checks explicitly, cf. [BGV, Proposition 3.58] $$\G \D \G =(\D')^t, \quad \G\D' \G=\D^t.$$
Recall that the maximal domain of $\D, \D'$ can also be characterized as a subspace of $L^2_*(M,E)$ with its image under $\D, \D'$ being again in $L^2_*(M,E)$. Since $\G$ gives an involution on $L^2_*(M,E)$, we obtain: 
\begin{align*}
\G \Da \G =(\D')^t_{\textup{max}}, \quad \G\Da' \G=\Da^t&, \\
\textup{i.e.} \quad \G \Da \G =(\Dr')^*, \quad \G\Da' \G=\Dr^*&.
\end{align*}
Taking adjoints on both sides of the last relation, we obtain the full statement of the lemma, since $\G$ is self-adjoint. 
\end{proof} \ \\
\\[-7mm] Now we can introduce the following central concepts.
\begin{defn}\label{domain}
$(\domm, \DD):=(\domr, \Dr)\oplus (\doma, \Da).$ The chirality operator $\GG$ on $(\domm, \DD)$ by definition acts anti-diagonally with respect to the direct sum of the components
\begin{align}\label{chirality}
\GG :=\left(\begin{array}{rr} 0 & \G \\ \G & 0 \end{array}\right).
\end{align}
\end{defn}\ \\
\\[-7mm] The Fredholm complex $(\domm, \DD)$ with the chirality operator $\GG$ is in case of a flat Hermitian metric a complex with Poincare duality, in the sense of [BL1, Lemma 2.16], i.e. $$\D h^E=0 \ \Rightarrow \ \GG\DD=\DD^*\GG,$$ which follows directly from Lemma \ref{G-Lemma}. We now apply the concepts of Braverman and Kappeler to our new setup.
\begin{defn}\label{odd-signature} The odd-signature operator of the Hilbert complex $(\domm, \DD)$ is defined as follows $$\B:=\GG\DD+\DD\GG.$$
\end{defn}\ \\
\\[-7mm] Before we can state some basic properties of the odd signature operator, let us recall the notions of the Gauss-Bonnet operator and its relative and absolute self-adjoint extensions. The Gauss-Bonnet operator 
$$D^{GB}:=\D+\D^t,$$ admits two natural self-adjoint extensions
\begin{align}\label{gauss-bonnet-rel-abs}
D^{GB}_{\textup{rel}}=\Dr+\Dr^*, \ D^{GB}_{\textup{abs}}=\Da+\Da^*,
\end{align}
respectively called the relative and the absolute self-adjoint extensions. Their squares are correspondingly the relative and the absolute Laplace operators:
$$\triangle_{\textup{rel}}=(D^{GB}_{\textup{rel}})^*D^{GB}_{\textup{rel}}, 
\quad \triangle_{\textup{abs}}=(D^{GB}_{\textup{abs}})^*D^{GB}_{\textup{abs}}.$$
Similar definitions, of course, hold for the Gauss-Bonnet Operator associated to the dual covariant derivative $\D'$. Now we can state the following basic result.
\begin{lemma}\label{odd-signature-laplacian} The leading symbols of $\B$ and $\GG \left(D^{GB}_{\textup{rel}}\oplus D'^{GB}_{\textup{abs}}\right)$ coincide and moreover
$$\dom(\B)= \dom \left(D^{GB}_{\textup{rel}}\oplus D'^{GB}_{\textup{abs}}\right).$$
\end{lemma}
\begin{proof} First recall the relations $$\G\D\G=(\D')^t, \quad \G\D^t\G=\D'.$$ All connections differ by an endomorphism-valued differential form of degree one, which can be viewed as a differential operator of order zero. This implies the statement on the leading symbol of $\B$ and $\GG \left(D^{GB}_{\textup{rel}}\oplus D'^{GB}_{\textup{abs}}\right)$
\\[3mm] A differential operator of zero order naturally extends to a bounded operator on the $L^2$-Hilbert space, and hence does not pose additional restrictions on the domain, in particular we obtain (compare Lemma \ref{G-Lemma})
$$\dom (\D_{\min}^*)=\dom (\G \D_{\max} \G), \quad \dom (\D_{\max}^*)=\dom (\G \D_{\min} \G).$$ Using these domain relations we find:
\begin{align*}
\dom(\B)= \dom \left(\GG(D^{GB}_{\textup{rel}}\oplus D'^{GB}_{\textup{abs}})\right)=\dom \left(D^{GB}_{\textup{rel}}\oplus D'^{GB}_{\textup{abs}}\right).
\end{align*}
\end{proof} \ \\
\\[-7mm] Note by the arguments of the lemma above that $\B$ is a bounded perturbation of a closed operator $\GG \left(D^{GB}_{\textup{rel}}\oplus D'^{GB}_{\textup{abs}}\right)$ and hence is closed, as well. Before we continue analyzing the spectral properties of the odd-signature operator $\B$, let us introduce some concepts and notation.
\begin{defn}
Let $D$ be a closed operator in a separable Hilbert space. An angle $\theta\in [0,2\pi)$ is called an "Agmon angle" for $D$, if for $R_{\theta}\subset \C$ being the cut in $\C$ corresponding to $\theta$ $$R_{\theta}:=\{z \in \C | z=|z|\cdot e^{i\theta}\}$$
we have the following spectral relation $$R_{\theta}\cap \textup{Spec}(D)\backslash \{0\}=\emptyset.$$
\end{defn}

\begin{thm}\label{Freddy} \textup{[S. Agmon, R. Seeley]}
Let $(K,g^K)$ be a smooth compact oriented Riemannian manifold with boundary $\partial K$. Let $(F,h^F)$ be a Hermitian vector bundle over $K$. The metric structures $(g^K,h^F)$ define an $L^2$-inner product. Let $$D:C^{\infty}(K,F)\to C^{\infty}(K,F)$$ be a differential operator of order $\w$ such that $\w \cdot \textup{rank}F$ is even. Consider a boundary value problem $(D,B)$ strongly elliptic with respect to $\C \backslash \R^*$ in the sense of [Gi]. Then
\begin{enumerate}
\item  $D_B$ is a Fredholm operator with compact resolvent and discrete spectrum of eigenvalues of finite (algebraic) multiplicity, accumulating only at infinity. 
\item The operator $D_B$ admits an Agmon angle $\theta \in (-\pi, 0)$ and the associated zeta-function 
\begin{align*}
&\zeta(s, D_B):=\sum\limits_{\lambda \in \textup{Spec}(D_B)\backslash \{0\}}m(\lambda)\cdot \lambda_{\theta}^{-s}, \quad \textup{Re}(s) > \frac{\dim K}{\w},
\end{align*}
where $\lambda_{\theta}^{-s}:=\textup{exp}(-s\cdot \log_{\theta}\lambda)$ and $m(\lambda)$ denotes the multiplicity of the eigenvalue $\lambda$,
is holomorphic for $\textup{Re}(s) > \dim K / \w$ and admits a meromorphic extension to the whole complex plane $\C$ with $s=0$ being a regular point. 
\end{enumerate}
\end{thm}\ \\ 
\\[-7mm] For the proof of the theorem note that the notion of strong ellipticity in the sense of [Gi] in fact combines ellipticity with Agmon's conditions, as in the treatment of elliptic boundary conditions by R.T. Seeley in [Se1, Se2]. The statement of the theorem above follows then from [Ag] and [Se1, Se2].
\begin{remark}
The definition of a zeta-function, as in Theorem \ref{Freddy} (ii), also applies to any operator $D$ with finite spectrum $\{\lambda_1,..,\lambda_n\}$ and finite respective multiplicities $\{m_1,..,m_n\}$. For a given Agmon angle $\theta \in [0,2\pi)$ the associated zeta-function $$\zeta_{\theta}(s,D):=\sum_{i=1, \lambda_i\neq 0}^nm_i\cdot (\lambda_i)^{-s}_{\theta}$$ is holomorphic for all $s\in \C$, since the sum is finite and the eigenvalue zero is excluded.
\end{remark}\ \\
\\[-7mm] Now we return to our specific setup. The following result is important in view of the relation between $\B$ and the Gauss-Bonnet operators with relative and absolute boundary conditions, as established in Lemma \ref{odd-signature-laplacian}.
\begin{prop}\label{strongly-elliptic0}
The operators 
$$D=\GG (D^{GB}_{\textup{rel}}\oplus D'^{GB}_{\textup{abs}}), \quad D^2=\triangle_{\textup{rel}}\oplus \triangle'_{\textup{abs}}$$
are strongly elliptic with respect to $\C\backslash \R^*$ and $\C\backslash \R^+$, respectively, in the sense of P. Gilkey [Gi]. 
\end{prop}\ \\
\\[-7mm] The fact that $D^2=\triangle_{\textup{rel}}\oplus \triangle'_{\textup{rel}}$ is strongly elliptic with respect to $\C\backslash \R^+$ is already encountered in Theorem \ref{thm41}. The strong ellipticity of $D$ now follows from [Gi, Lemma 1.11.2]. Note that this result in the reference [Gi] is proved explicitly, even though other aspects of [Gi, Section 1.11] are rather expository.
\\[3mm] Since Lemma \ref{odd-signature-laplacian} asserts the equality between the leading symbols of the differential operators $\B,D$ and moreover the equality of the associated boundary conditions, the odd signature operator $\B$ and its square $\B^2$ are strongly elliptic as well. This proves together with Theorem \ref{Freddy} the next proposition.
\begin{prop}\label{strongly-elliptic}
The operators $\B$ and $\B^2$ are strongly elliptic with respect to $\C\backslash \R^*$ and $\C\backslash \R^+$, respectively, in the sense of P. Gilkey [Gi]. The operators $\B, \B^2$ are discrete with their spectrum accumulating only at infinity. 
\end{prop} \ \\
\\[-7mm] Let now $\lambda \geq 0$ be any non-negative real number. Denote by $\Pi_{\B^2, [0,\lambda]}$ the spectral projection of $\B^2$ onto eigenspaces with eigenvalues of absolute value in the interval $[0,\lambda]$:
$$\Pi_{\B^2, [0,\lambda]}:=\frac{i}{2\pi}\int_{C(\lambda)}(\B^2-x)^{-1}dx,$$
with $C(\lambda)$ being any closed counterclockwise circle surrounding eigenvalues of absolute value in $[0,\lambda]$ with no other eigenvalue inside. One finds using the analytic Fredholm theorem that the range of the projection lies in the domain of $\B^2$ and that the projection commutes with $\B^2$.
\\[3mm] Since $\B^2$ is discrete, the spectral projection $\Pi_{\B^2, [0,\lambda]}$ is of finite rank, i.e. with a finite-dimensional image. In particular $\Pi_{\B^2,[0,\lambda]}$ is a bounded operator in $L^2_*(M,E\oplus E)$.  Hence with [K, Section 4, p.155] the decomposition
\begin{align}\label{decomp-L-2}
L^2_*(M,E\oplus E)=\textup{Image}\Pi_{\B^2, [0,\lambda]}\oplus \textup{Image}(\one - \Pi_{\B^2, [0,\lambda]}),
\end{align}
is a direct sum decomposition into closed subspaces of the Hilbert space $L^2_*(M,E\oplus E)$.
\\[3mm] Note that if $\B^2$ is self-adjoint, the decomposition is orthogonal with respect to the fixed $L^2-$Hilbert structure, i.e. the projection $\Pi_{\B^2,[0,\lambda]}$ is an orthogonal projection, which is the case only if the Hermitian metric $h^E$ is flat with respect to $\nabla$.
\\[3mm] The decomposition induces by restriction a decomposition of $\domm$, which was introduced in Definition \ref{domain}:
$$\domm=\domm_{[0,\lambda]}\oplus \domm_{(\lambda, \infty)}.$$ Since $\DD$ commutes with $\B, \B^2$ and hence also with $\Pi_{\B^2, [0,\lambda]}$, we find that the decomposition above is in fact a decomposition into subcomplexes:
\begin{align}\nonumber
(\domm, \DD)=(\domm_{[0,\lambda]}, \DD_{[0,\lambda]})\oplus (\domm_{(\lambda, \infty)}, \DD_{(\lambda, \infty)}) \\ \label{decomposition}
\textup{where} \ \DD_{\mathcal{I}}:=\DD|_{\domm_{\mathcal{I}}} \ \textup{for} \ \mathcal{I}=[0,\lambda] \ \textup{or} \ (\lambda, \infty).
\end{align}
Further $\GG$ also commutes with $\B, \B^2$ and hence also with $\Pi_{\B^2, [0,\lambda]}$. Thus as above we obtain $$\GG = \GG_{[0,\lambda]} \oplus \GG_{(\lambda, \infty)}.$$ Consequently the odd-signature operator of the complex $(\domm, \DD)$ decomposes correspondingly 
\begin{align}\nonumber
&\B=\B^{[0,\lambda]}\oplus \B^{(\lambda, \infty)}\\ 
\textup{where} \quad  &\B^{\mathcal{I}}:=\GG_{\mathcal{I}}\DD_{\mathcal{I}}+\DD_{\mathcal{I}}\GG_{\mathcal{I}} \ \textup{for} \ \mathcal{I}=[0,\lambda] \ \textup{or} \ (\lambda, \infty). \label{66}
\end{align}
The closedness of the subspace Image$(1-\Pi_{\B^2,[0,\lambda]})$ implies that the domain of $\B^{(\lambda, \infty)}$ 
$$\dom (\B^{(\lambda, \infty)}):=\dom (\B)\cap \textup{Image}(1-\Pi_{\B^2,[0,\lambda]})$$
is closed under the graph-norm, hence the operator $\B^{(\lambda, \infty)}$ is a closed operator in the Hilbert space $\textup{Image}(1-\Pi_{\B^2,[0,\lambda]})$.
\\[3mm] We need to analyze the direct sum component $\B^{(\lambda, \infty)}$. For this we proceed with the following general functional analytic observations.
\begin{prop}\label{compact}
Let $D$ be a closed operator in a separable Hilbert space $(H, \langle \cdot ,\cdot \rangle)$. The domain $\dom (D)$ is a Hilbert space with the graph-norm $$\langle x,y\rangle_D=\langle x,y\rangle+\langle Dx,Dy\rangle$$ for any $x,y \in \dom (D)$. Let Res$D \neq \emptyset$. Then the following statements are equivalent \\
1) \ The inclusion $\iota : \dom (D) \hookrightarrow H$ is a compact operator \\
2) \ $D$ has a compact resolvent, i.e. for some (and thus for all) $z \in$ Res$(D)$ the resolvent operator $(D-z)^{-1}$ is a compact operator on $H$.
\end{prop}
\begin{proof}
Assume first that the inclusion $\iota : \dom (D) \hookrightarrow H$ is a compact operator. Since Spec$D \neq \C$ the resolvent set Res$(D)$ is not empty. For any $z \in$ Res$(D)$ the resolvent operator $$(D-z)^{-1}: H \to \dom (D)$$ exists and is bounded, by definition of the resolvent set. With the inclusion $\iota$ being a compact operator we find directly that $(D-z)^{-1}$ is compact as an operator from $H$ to $H$. Finally, if $(D-z)^{-1}$ is compact for some $z \in$ Res$(D)$, then by the second resolvent identity it is compact for all $z \in$ Res$(D)$, see also [K, p.187].
\\[3mm] Conversely assume that for some (and therefore for all) $z \in$ Res$(D)$ the resolvent operator $(D-z)^{-1}$ is compact as an operator from $H$ into $H$. Observe $$\iota = (D-z)^{-1}\circ (D-z):\dom (D) \hookrightarrow H. $$ By compactness of the resolvent operator, $\iota$ is compact as an operator between the Hilbert spaces $\dom (D)$ and $H$.
\end{proof}

\begin{prop}\label{index-zero}
Let $D$ be a closed operator in a separable Hilbert space $H$ with Res$(D) \neq \emptyset$ and compact resolvent. Then $D$ is a Fredholm operator with $$\textup{index} \, D=0.$$
\end{prop}
\begin{proof}
By closedness of $D$ the domain $\dom(D)$ turns into a Hilbert space equipped with the graph norm. By Proposition \ref{compact} the natural inclusion  $$\iota : \dom (D) \hookrightarrow H$$ is a compact operator. Therefore, viewing $\dom (D)$ as a subspace of $H$, i.e. endowed with the inner-product of $H$, the inclusion $$\iota : \dom (D) \subset H \hookrightarrow H$$ is relatively $D$-compact in the sense of [K, Section 4.3, p.194]. More precisely this means, that if for a sequence $\{u_n\}\subset \dom (D)$ both $\{u_n\}$ and $\{Du_n\}$ are bounded sequences in $H$, then $\{\iota (u_n)\}\subset H$ has a convergent subsequence.
\\[3mm] Now for any $\lambda \in \C\backslash \textup{Spec}(D)$ the operator $$(D-\lambda \iota):\dom (D) \subset H \rightarrow H$$ is invertible and hence trivially a Fredholm operator with trivial kernel and closed range $H$. In particular $$\textup{index}(D-\lambda \iota)=0.$$
Now, from stability of the Fredholm index under relatively compact perturbations (see [K, Theorem 5.26] and the references therein) we infer with the inclusion $\iota$ being relatively compact, that $D$ is a Fredholm operator of zero index:$$\textup{index}\, D=\textup{index}(D-\lambda\iota)=0.$$
\end{proof}

\begin{cor}\label{bijective} The operator $\B^{(\lambda, \infty)}: \dom (\B^{(\lambda, \infty)})\to \textup{Image}(1-\Pi_{\B^2,[0,\lambda]})$ of the complex $(\domm_{(\lambda, \infty)}, \DD_{(\lambda, \infty)})$ with $\lambda \geq 0$ is bijective.
\end{cor}
\begin{proof} 
Consider any $\lambda \in \C \backslash \textup{Spec}\B$. By the strong ellipticity of $\B$, the operator $$(\B-\lambda):\dom (\B)\rightarrow L^2_*(M,E\oplus E)$$ is bijective with compact inverse. Hence we immediately find that the restriction 
\begin{align*}
(\B^{(\lambda, \infty)}-\lambda)\equiv (\B-\lambda)\restriction \textup{Im}(1-\Pi_{\B^2,[0,\lambda]}): \dom (\B^{(\lambda, \infty)})\rightarrow \textup{Im}(1-\Pi_{\B^2,[0,\lambda]})
\end{align*}
is bijective with compact inverse, as well. Now we deduce from Proposition \ref{index-zero} that $\B^{(\lambda, \infty)}$ is Fredholm with $$\textup{index}\, \B^{(\lambda, \infty)}=0.$$
The operator $\B^{(\lambda, \infty)}$ is injective, by definition. Combining injectivity with the vanishing of the index, we derive surjectivity of $\B^{(\lambda, \infty)}$. This proves the statement.
\end{proof}\ \\
\\[-7mm] Note, that in case of a flat Hermitian metric the assertion of the previous corollary is simply the general fact that a self-adjoint Fredholm operator is invertible if and only if its kernel is trivial. 
\begin{cor}\label{cohomology} The subcomplex $(\domm_{(\lambda, \infty)}, \DD_{(\lambda, \infty)})$ is acyclic and $$H^*((\domm_{[0,\lambda]}, \DD_{[0,\lambda]}))\cong H^*(\domm, \DD).$$
\end{cor}
\begin{proof}
Corollary \ref{bijective} allows us to apply the purely algebraic result [BK2, Lemma 5.8]. Consequently $(\domm_{(\lambda, \infty)}, \DD_{(\lambda, \infty)})$ is an acyclic complex. Together with the decomposition \eqref{decomposition} this proves the assertion.
\end{proof} \ \\
\\[-7mm] Observe that since the spectrum of $\B^2$ is discrete accumulating only at infinity, $(\domm_{[0,\lambda]}, \DD_{[0,\lambda]})$ is a complex of finite-dimensional complex vector spaces with $\GG_{[0,\lambda]}:\domm^k_{[0,\lambda]}\to \domm^{m-k}_{[0,\lambda]}$ being the chirality operator on the complex in the sense of [BK2, Section 1.1]. 
\\[3mm] We also use the notion of determinant lines of finite dimensional complexes in [BK2, Section 1.1], which are given for any finite complex of finite-dimensional vector spaces $(C^*,\partial_*)$ as follows:
$$\textup{Det}H^*(C^*,\partial_*)=\bigotimes\limits_k \det H^k(C^*,\partial_*)^{(-1)^k}, $$
where $\det H^k(C^*,\partial_*)$ is the top exterior power of $H^k(C^*,\partial_*)$ and $\det H^k(C^*,\partial_*)^{-1}\equiv\det H^k(C^*,\partial_*)^*$. We follow [BK2, Section 1.1] and form the "refined torsion" (note the difference to "refined analytic torsion") of the complex $(\domm_{[0,\lambda]}, \DD_{[0,\lambda]})$
\begin{align}\label{finite-torsion}
\rho_{[0,\lambda]}:=c_0\otimes (c_1)^{-1}\otimes \cdots \otimes (c_r)^{(-1)^r} \otimes (\GG_{[0,\lambda]}c_r)^{(-1)^{r+1}}\otimes \cdots \\ \cdots \otimes (\GG_{[0,\lambda]}c_1) \otimes (\GG_{[0,\lambda]}c_0)^{(-1)}\in \textup{Det}(H^*(\domm_{[0,\lambda]}, \DD_{[0,\lambda]})), \nonumber 
\end{align}
where $c_k\in \det H^k(\domm_{[0,\lambda]}, \DD_{[0,\lambda]})$ are arbitrary elements of the determinant lines, $\GG_{[0,\lambda]}$ denotes the chirality operator $\GG_{[0,\lambda]}:\domm^{\bullet}_{[0,\lambda]}\to \domm^{m-\bullet}_{[0,\lambda]}$ extended to determinant lines and for any $v\in \det H^k(\domm_{[0,\lambda]}, \DD_{[0,\lambda]})$ the dual $v^{-1}\in \det H^k(\domm_{[0,\lambda]}, \DD_{[0,\lambda]})^{-1}\equiv \det H^k(\domm_{[0,\lambda]}, \DD_{[0,\lambda]})^*$ is the unique element such that $v^{-1}(v)=1$.
\\[3mm] By Corollary \ref{cohomology} we can view $\rho_{[0,\lambda]}$ canonically as an element of $\textup{Det}(H^*(\domm, \DD))$, which we do henceforth.
\\[3mm] The second part of the construction is the graded determinant. The operator $\B^{(\lambda, \infty)},\lambda \geq 0$ is bijective by Corollary \ref{bijective} and hence by injectivity (put $\mathcal{I}=(\lambda, \infty)$ to simplify the notation) 
\begin{align}\label{kern}
\textup{ker}(\DD_{\mathcal{I}}\GG_{\mathcal{I}})\cap\textup{ker}(\GG_{\mathcal{I}}\DD_{\mathcal{I}})=\{0\}.
\end{align}
Further the complex $(\domm_{\mathcal{I}}, \DD_{\mathcal{I}})$ is acyclic by Corollary \ref{cohomology} and due to $\GG_{\mathcal{I}}$ being an involution on $\textup{Im}(1-\Pi_{\B^2,[0,\lambda]})$ we have
\begin{align}\label{image1}
\textup{ker}(\DD_{\mathcal{I}}\GG_{\mathcal{I}})=\GG_{\mathcal{I}}\textup{ker}(\DD_{\mathcal{I}})=\GG_{\mathcal{I}}\textup{Im}(\DD_{\mathcal{I}})=\textup{Im}(\GG_{\mathcal{I}}\DD_{\mathcal{I}}), \\ \label{image2} 
\textup{ker}(\GG_{\mathcal{I}}\DD_{\mathcal{I}})=\textup{ker}(\DD_{\mathcal{I}})=\textup{Im}(\DD_{\mathcal{I}})=\textup{Im}(\DD_{\mathcal{I}}\GG_{\mathcal{I}}).
\end{align}
We have $\textup{Im}(\GG_{\mathcal{I}}\DD_{\mathcal{I}})+\textup{Im}(\DD_{\mathcal{I}}\GG_{\mathcal{I}})=\textup{Im}(\B^{\mathcal{I}})$ and by surjectivity of $\B^{\mathcal{I}}$ we obtain from the last three relations above
\begin{align}\label{hilbert-decomposition}
\textup{Im}(1-\Pi_{\B^2, [0,\lambda]})=\textup{ker}(\DD_{\mathcal{I}}\GG_{\mathcal{I}})\oplus\textup{ker}(\GG_{\mathcal{I}}\DD_{\mathcal{I}}).
\end{align}
Note that $\B$ leaves $\ker (\DD\GG)$ and $\ker (\GG\DD)$ invariant. Put 
\begin{align*}
\B^{+,(\lambda,\infty)}_{\textup{even}}:=\B^{(\lambda,\infty)}\restriction \domm^{\textup{even}}\cap \ker (\DD\GG), \\
\B^{-,(\lambda,\infty)}_{\textup{even}}:=\B^{(\lambda,\infty)}\restriction \domm^{\textup{even}}\cap \ker (\GG\DD).
\end{align*}
We obtain a direct sum decomposition $$\B^{(\lambda,\infty)}_{\textup{even}}=\B^{+,(\lambda,\infty)}_{\textup{even}} \oplus \B^{-,(\lambda,\infty)}_{\textup{even}}.$$
As a consequence of Theorem \ref{Freddy} (ii) and Proposition \ref{strongly-elliptic} there exists an Agmon angle $\theta\in (-\pi, 0)$ for $\B$, which is clearly an Agmon angle for the restrictions above, as well. 
\\[3mm] By Theorem \ref{Freddy} and Proposition \ref{strongly-elliptic} the zeta function $\zeta_{\theta}(s,\B)$ is holomorphic for Re$(s)$ sufficiently large. The zeta-functions $\zeta_{\theta}(s,\B^{\pm,(\lambda,\infty)}_{\textup{even}})$ of $\B^{\pm,(\lambda,\infty)}_{\textup{even}}$, defined with respect to the given Agmon angle $\theta$, are holomorphic for Re$(s)$ large as well, since the restricted operators have the same spectrum as $\B$ but in general with lower or at most the same multiplicities. 
\\[3mm] We define the \emph{graded zeta-function} $$\zeta_{gr,\theta}(s,\B^{(\lambda,\infty)}_{\textup{even}}):=\zeta_{\theta}(s,\B^{+,(\lambda,\infty)}_{\textup{even}})-\zeta_{\theta}(s,-\B^{-,(\lambda,\infty)}_{\textup{even}}), \ Re(s)\gg 0.$$
\\[3mm] In the next subsection we prove in Theorem \ref{log-det-gr} that the graded zeta-function extends meromorphically to $\C$ and is regular at $s=0$. For the time being we shall assume regularity at zero and define the graded determinant.
\begin{defn}\label{graded-determinant}[Graded determinant] Let $\theta \in (-\pi, 0)$ be an Agmon angle for $\B^{(\lambda, \infty)}$. Then the "graded determinant" associated to $\B^{(\lambda, \infty)}$ and its Agmon angle $\theta$ is defined as follows: $$\det\nolimits_{gr,\theta}(\B^{(\lambda, \infty)}_{\textup{even}}):= \textup{exp}(-\left.\frac{d}{ds}\right|_{s=0} \zeta_{gr,\theta}(s,\B^{(\lambda,\infty)}_{\textup{even}})).$$
\end{defn}
\begin{prop}\label{rho-element} The element $$\rho(\D, g^M, h^E):=\det\nolimits_{gr,\theta}(\B^{(\lambda, \infty)}_{\textup{even}})\cdot \rho_{[0,\lambda]}\in \textup{Det}(H^*(\domm, \DD))$$ is independent of the choice of $\lambda \geq 0$ and choice of Agmon angle $\theta \in (-\pi, 0)$ for the odd-signature operator $\B^{(\lambda, \infty)}$. 
\end{prop}
\begin{proof}
Let $0 \leq \lambda < \mu < \infty$. We obtain $\domm_{[0,\mu]}=\domm_{[0,\lambda]}\oplus \domm_{(\lambda, \mu]}$ and also $\domm_{(\lambda, \infty)}\!=\domm_{(\lambda, \mu]}\oplus \domm_{(\mu, \infty)}$. Since the odd-signature operator respects this spectral direct sum decomposition (see \eqref{66}), we obtain
$$\det\nolimits_{gr}(\B^{(\lambda, \infty)}_{\textup{even}})=\det\nolimits_{gr}(\B^{(\mu, \infty)}_{\textup{even}})\cdot \det\nolimits_{gr}(\B^{(\lambda, \mu]}_{\textup{even}}).$$
Further the purely algebraic discussion behind [BK2, Proposition 5.10] implies $$\rho_{[0,\mu]}=\det\nolimits_{gr}(\B^{(\lambda, \mu]}_{\textup{even}})\cdot \rho_{[0,\lambda]}.$$ This proves the following equality $$\det\nolimits_{gr}(\B^{(\lambda, \infty)}_{\textup{even}})\cdot \rho_{[0,\lambda]}=\det\nolimits_{gr}(\B^{(\mu, \infty)}_{\textup{even}})\cdot \rho_{[0,\mu]}.$$ 
To see independence of $\theta \in (-\pi, 0)$ note that the strongly elliptic operator (cf. Lemma \ref{odd-signature-laplacian}) $$D:=\GG (D^{GB}_{rel}\oplus D'^{GB}_{abs})$$ is self-adjoint and $\B$ differs from $D$ by a bounded perturbation. By a Neumann-series argument and the asymptotics of the resolvent for $D$ (see [Se1, Lemma 15]) we get: 
\begin{align}\label{spectral-cut}
\forall \theta \in (-\pi, 0): \quad \textup{Spec}(\B)\cap R_{\theta} \quad \textup{is finite.}
\end{align}
By discreteness of $\B$ we deduce that if $\theta, \theta' \in (-\pi,0)$ are both Agmon angles for $\B^{(\lambda, \infty)}$, there are only finitely many eigenvalues of $\B^{(\lambda, \infty)}$ in the solid angle between $\theta$ and $\theta'$. Hence 
\begin{align*}
\left.\frac{d}{ds}\right|_{s=0}\zeta_{gr,\theta}(s,\B^{(\lambda,\infty)}_{\textup{even}}))\equiv \left.\frac{d}{ds}\right|_{s=0}\zeta_{gr,\theta'}(s,\B^{(\lambda,\infty)}_{\textup{even}})) \quad \textup{mod} \ 2\pi i, \\
\textup{and therefore } \ \det\nolimits_{gr,\theta}(\B^{(\lambda, \infty)}_{\textup{even}})=\det\nolimits_{gr,\theta'}(\B^{(\lambda, \infty)}_{\textup{even}}).
\end{align*}
This proves independence of the choice of $\theta \in (-\pi, 0)$ and completes the proof.
\end{proof}\ \\
\\[-7mm] The element $\rho(\D, g^M,h^E)$ is well-defined but a priori not independent of the choice of metrics $g^M, h^E$ and so does not provide a differential invariant. In the next subsection we determine the metric anomaly of $\rho(\D, g^M,h^E)$ in order to construct a differential invariant, which will be called the refined analytic torsion.

\section{Metric Anomaly and Refined Analytic Torsion}\label{anomaly} 
We introduce the notion of the eta-function leading to the notion of the eta-invariant of an elliptic operator. The eta-invariant was first introduced by Atiyah-Patodi-Singer in [APS] as the boundary correction term in their index formula.

\begin{thm}\label{eta-regular-original}[P.B. Gilkey, L. Smith] Let $(K,g^K)$ be a smooth compact oriented Riemannian manifold with boundary $\partial K$. Let $(F,h^F)$ be a Hermitian vector bundle and let the metric structures $(g^K,h^F)$ define an $L^2-$scalar product. Let $$D: C^{\infty}(K,F)\rightarrow C^{\infty}(K,F)$$ be a differential operator of order $\w$ such that $\w \cdot \textup{rank}F$ is even. Let a boundary value problem $(D,B)$ be strongly elliptic with respect to $\C \backslash \R^*$ and an Agmon angle $\theta \in (-\pi, 0)$. Then we have
\begin{enumerate}
\item $D_B$ is a discrete Fredholm operator in the Hilbert space $L^2(K,F)$ and its eta-function $$\eta_{\theta}(s,D_B):=\sum\limits_{\textup{Re}(\lambda)>0}m(\lambda)\cdot \lambda_{\theta}^{-s}-\sum\limits_{\textup{Re}(\lambda)<0}m(\lambda)\cdot(-\lambda)_{\theta}^{-s},$$ where $m(\lambda)$ denotes the finite (algebraic) multiplicity of the eigenvalue $\lambda$ , is holomorphic for Re$(s)$ large and extends meromorphically to $\C$ with at most simple poles. \\
\item If $D$ is of order one with the leading symbol $\sigma_D(x,\xi), x \in K, \xi \in T^*_xK$ satisfying $$\sigma_D(x,\xi)^2=|\xi|^2\cdot I,$$ where $I$ is $\textup{rank}F\times \textup{rank}F$ identity matrix, and the boundary condition $B$ is of order zero, then the meromorphic extension of $\eta_{\theta}(s,D_B)$ is regular at $s=0$.
\end{enumerate}
\end{thm} \ \\
\\[-7mm] The proof of the theorem follows from the results in [GS1] and [GS2] on the eta-function of strongly elliptic boundary value problems. The fact that $\eta_{\theta}(s,D_B)$ is holomorphic for Re$(s)$ sufficiently large is asserted in [GS1, Lemma 2.3 (c)]. The meromorphic continuation with at most isolated simple poles is asserted in [GS1, Theorem 2.7]. 
\\[3mm] The fact that $s=0$ is a regular point of the eta-function is highly non-trivial and cannot be proved by local arguments. Using homotopy invariance of the residue at zero for the eta-function, P. Gilkey and L. Smith [GS2] reduced the discussion to a certain class of operators with constant coefficients in the collar neighborhood of the boundary and applied the closed double manifold argument. The reduction works for differential operators of order one with 0-th order boundary conditions under the assumption on the leading symbol of the operator as in the second statement of the theorem. The regularity statement of Theorem \ref{eta-regular-original} follows directly from [GS2, Theorem 2.3.5] and [GS2, Lemma 2.3.4]. 

\begin{remark}
The definition of an eta-function, as in Theorem \ref{eta-regular-original} $(i)$, also applies to any operator $D$ with finite spectrum $\{\lambda_1,..,\lambda_n\}$ and finite respective multiplicities $\{m_1,..,m_n\}$. For a given Agmon angle $\theta \in [0,2\pi)$ the associated eta-function $$\eta_{\theta}(s,D):=\sum\limits_{\textup{Re}(\lambda)>0}m(\lambda)\cdot \lambda_{\theta}^{-s}-\sum\limits_{\textup{Re}(\lambda)<0}m(\lambda)\cdot(-\lambda)_{\theta}^{-s},$$ is holomorphic for all $s\in \C$, since the sum is finite and the zero-eigenvalue is excluded.
\end{remark}

\begin{prop}\label{eta-regular}
The eta-function $\eta_{\theta}(s,\B_{\textup{even}})$ associated to the even part $\B_{\textup{even}}$ of the odd-signature operator and its Agmon angle $\theta \in (-\pi,0)$, is holomorphic for Re$(s)$ large and extends meromorphically to $\C$ with $s=0$ being a regular point. 
\end{prop}\ \\
\\[-7mm] The statement of the proposition on the meromorphic extension of the eta-function is a direct consequence of Theorem \ref{eta-regular-original} (i) and Proposition \ref{strongly-elliptic}. The regularity statement follows from Theorem \ref{eta-regular-original} (ii) and an explicit computation of the leading symbol of the odd-signature operator, compare also [GS2, Example 2.2.4].
\\[3mm] Using Proposition \ref{eta-regular} we can define the eta-invariant in the manner of [BK2] for $\B_{\textup{even}}$:
\begin{align}\label{eta-BK}
\eta(\B_{\textup{even}}):=\frac{1}{2}\left(\eta_{\theta}(s=0,\B_{\textup{even}})+m_+-m_-+m_0\right),
\end{align}
where $m_{\pm}$ is the number of $\B_{\textup{even}}-$eigenvalues on the positive, respectively the negative part of the imaginary axis and $m_0$ is the dimension of the generalized zero-eigenspace of $\B_{\textup{even}}$.
\\[3mm] Implicit in the notation is also the fact, that $\eta(\B_{\textup{even}})$ does not depend on the Agmon angle $\theta \in (-\pi, 0)$. This is due to the fact that, given a different Agmon angle $\theta'\in (-\pi, 0)$, there are by \eqref{spectral-cut} and discreteness of $\B$ only finitely many eigenvalues of $\B_{\textup{even}}$ in the acute angle between $\theta$ and $\theta'$. 
\\[3mm] Similarly we define the eta-invariants of $\B^{(\lambda, \infty)}_{\textup{even}}$ and $\B^{[0,\lambda]}_{\textup{even}}$ and in particular we get 
$$\eta(\B_{\textup{even}})=\eta(\B^{(\lambda, \infty)}_{\textup{even}})+ \eta(\B^{[0,\lambda]}_{\textup{even}}).$$
Before we prove the next central result, let us make the following observation.
\\[3mm] Consider the imaginary axis $i\R\subset \C$. By \eqref{spectral-cut} there are only finitely many eigenvalues of $\B$ on $i\R$. Further by the discreteness of $\B$ small rotation of the imaginary axis does not hit any further eigenvalue of $\B$ and in particular of $\B^{(\lambda,\infty)}_{\textup{even}}, \lambda \geq 0$. More precisely this means that there exists an $\epsilon > 0$ sufficiently small such that the angle $$\theta:=-\frac{\pi}{2}+\epsilon $$ is an Agmon angle for $\B^{(\lambda,\infty)}_{\textup{even}}$ and the solid angles 
\begin{align*}
L_{(-\pi /2, \theta]}&:=\{z \in \C | z=|z|\cdot e^{i\phi}, \phi \in (-\pi /2, \theta]\}, \\
L_{(\pi /2, \theta+\pi]}&:=\{z \in \C | z=|z|\cdot e^{i\phi}, \phi \in (\pi /2, \theta+\pi]\}
\end{align*}
do not contain eigenvalues of $\B^{(\lambda,\infty)}_{\textup{even}}$. With this observation we can state the following central result:

\begin{thm}\label{log-det-gr}
Let $\theta \in (-\pi /2 , 0)$ be an Agmon angle for $\B^{(\lambda, \infty)}_{\textup{even}}$ such that there are no eigenvalues of $\B^{(\lambda, \infty)}_{\textup{even}}$ in the solid angles $L_{(-\pi /2, \theta]}$ and $L_{(-\pi /2 , \theta + \pi]}$. Then $2\theta$ is an Agmon angle for $(\B^{(\lambda,\infty)}_{\textup{even}})^2$. Then the graded zeta-function $\zeta_{gr,\theta}(s, \B^{(\lambda, \infty)}_{\textup{even}}), Re(s)\gg0$ extends meromorphically to $\C$ and is regular at $s=0$ with the following derivative at zero:
\begin{align*} \left.\frac{d}{ds}\right|_{s=0}\zeta_{gr,\theta}(s,\B^{(\lambda,\infty)}_{\textup{even}}))=\frac{1}{2}\sum_{k=0}^m(-1)^{k+1}\cdot k\cdot \left.\frac{d}{ds}\right|_{s=0}\zeta_{2\theta}(s, \B^2\restriction \domm^k_{(\lambda, \infty)}) + \\ + \frac{i \pi}{2}\sum_{k=0}^m(-1)^k\cdot k\cdot \zeta_{2\theta}(0, \B^2\restriction \domm^k_{(\lambda, \infty)}) + i\pi \eta(\B^{(\lambda, \infty)}_{\textup{even}}).
\end{align*}
\end{thm}
\begin{proof} For $Re(s)\gg0$ the general identities [BK1 (4.10), (4.11)] imply the following relation between holomorphic functions:
\begin{align*} \zeta_{gr,\theta}(s,\B^{(\lambda,\infty)}_{\textup{even}}))=\frac{1+e^{-i\pi s}}{2}\left[
\zeta_{2\theta}\left(\frac{s}{2}, \left(\B^{+,(\lambda, \infty)}_{\textup{even}}\right)^2\right)- \zeta_{2\theta}\left(\frac{s}{2}, \left(\B^{-,(\lambda, \infty)}_{\textup{even}}\right)^2\right)\right] + \\ +\frac{1}{2}(1-e^{-i\pi s}) \left[\eta(s, \B^{(\lambda, \infty)}_{\textup{even}})+f(s)\right],
\end{align*}
where $f(s)$ is a holomorphic function (combination of zeta-functions associated to finite-dimensional operators) with $$f(0)=m_+(\B^{(\lambda, \infty)}_{\textup{even}})-m_-(\B^{(\lambda, \infty)}_{\textup{even}}),$$ where $m_{\pm}(\cdot)$ denotes the number of eigenvalues of the operator in brackets, lying on the positive, respectively the negative part of the imaginary axis.
\\[3mm] Put $\mathcal{I}=(\lambda, \infty)$ to simplify notation. Recall \eqref{image2} and show that
\begin{align}\label{bijective-two}
\DD_{\mathcal{I}}: \textup{ker}(\DD_{\mathcal{I}}\GG_{\mathcal{I}})\rightarrow \textup{ker}(\GG_{\mathcal{I}}\DD_{\mathcal{I}})= \textup{Im}(\DD_{\mathcal{I}}\GG_{\mathcal{I}})
\end{align}
is bijective. Indeed, injectivity is clear by \eqref{kern}. For surjectivity let $x=\DD_{\mathcal{I}}\GG_{\mathcal{I}}v\in \textup{Im}(\DD_{\mathcal{I}}\GG_{\mathcal{I}})$ with (recall \eqref{hilbert-decomposition}) $$v=v'\oplus v'' \in \textup{Im}(\DD_{\mathcal{I}}\GG_{\mathcal{I}}) \oplus \textup{Im}(\GG_{\mathcal{I}}\DD_{\mathcal{I}})=\textup{Im}(1-\Pi_{\B^2,[0,\lambda]}).$$
In particular $v''\in \textup{Im}(\GG_{\mathcal{I}}\DD_{\mathcal{I}})=\ker \DD_{\mathcal{I}}\GG_{\mathcal{I}}$ and $v'=\DD_{\mathcal{I}}\GG_{\mathcal{I}} \w$ for some $\w$. Hence we obtain 

\begin{align*}
x=\DD_{\mathcal{I}}\GG_{\mathcal{I}}v=\DD_{\mathcal{I}}\GG_{\mathcal{I}}v'=\DD_{\mathcal{I}}\GG_{\mathcal{I}}\DD_{\mathcal{I}}\GG_{\mathcal{I}}\w, \\
\textup{and} \quad \GG_{\mathcal{I}}\DD_{\mathcal{I}}\GG_{\mathcal{I}}\w\in \ker\DD_{\mathcal{I}}\GG_{\mathcal{I}}. 
\end{align*}
In other words we have found a preimage of any $x\in \textup{Im}(\DD_{\mathcal{I}}\GG_{\mathcal{I}})$ under $\DD_{\mathcal{I}}$. This proves bijectivity of the map in \eqref{bijective-two} and consequently, since $\DD_{\mathcal{I}}$ commutes with $\B^{\mathcal{I}}$ and $(\B^{\mathcal{I}})^2$, we obtain in any degree $k=0,..,m$
\begin{align}\label{zwei}
\zeta_{2\theta}(s,(\B^{+,\mathcal{I}})^2 \restriction \domm^k)=\zeta_{2\theta}(s,(\B^{-, \mathcal{I}})^2 \restriction \domm^{k+1}).
\end{align}
Using this relation we compute straightforwardly for $Re(s)$ sufficiently large:
\begin{align}
\zeta_{2 \theta}(s,(\B^{+,\mathcal{I}}_{\textup{even}})^2)-\zeta_{2 \theta}(s,(\B^{-,\mathcal{I}}_{\textup{even}})^2)=\sum_{k=0}^m (-1)^{k+1} \cdot k\cdot \zeta_{2 \theta}(s,(\B^{\mathcal{I}})^2 \restriction \domm^k).
\end{align}
We arive at the following preliminary result for $Re(s)\gg0$ 
\begin{align}\label{graded-formula}
\zeta_{gr,\theta}(s,\B^{\mathcal{I}}_{\textup{even}}))=\frac{1}{2}(1+e^{-i\pi s})\sum_{k=0}^m(-1)^{k+1}\cdot k\cdot \zeta_{2\theta}(s,(\B^{\mathcal{I}})^2\restriction \domm^k) + \\ +\frac{1}{2}(1-e^{-i\pi s}) \left[\eta(s,\B^{\mathcal{I}}_{\textup{even}})+f(s)\right]. \nonumber
\end{align}
We find with Theorem \ref{Freddy} and Proposition \ref{eta-regular} that the right hand side of the equality above is a meromorphic function on the entire complex plane and is regular at $s=0$. Hence the left hand side of the equality, the graded zeta-function, is meromorphic on $\C$ and regular at $s=0$, as claimed and as anticipated in Definition \ref{graded-determinant}. Computing the derivative at zero, we obtain the statement of the theorem.
\end{proof}\ \\
\\[-7mm] As a consequence of the theorem above, we obtain for the element $\rho(\D, g^M,h^E)$ defined in Proposition \ref{rho-element} the following relation
\begin{align}\label{drei}
\rho(\D, g^M,h^E)&=e^{\xi_{\lambda}(\D, g^M)}e^{-i\pi \xi'_{\lambda}(\D, g^M)}e^{-i\pi \eta(\B^{(\lambda, \infty)}_{\textup{even}}(g^M))}\cdot \rho_{[0,\lambda]}, \\ \label{xi1}
\xi_{\lambda}(\D, g^M)&=\frac{1}{2}\sum_{k=0}^m(-1)^{k}\cdot k\cdot \left. \frac{d}{ds}\right|_{s=0}\zeta_{2\theta}(s,(\B^2\restriction \domm^k_{(\lambda, \infty)})) \\  \label{xi2} \xi'_{\lambda}(\D, g^M)&=\frac{1}{2}\sum_{k=0}^m(-1)^{k}\cdot k\cdot \zeta_{2\theta}(s=0,(\B^2\restriction \domm^k_{(\lambda, \infty)})). 
\end{align}
Now we can identify explicitly the metric dependence of $\rho(\D, g^M,h^E)$ using the formula \eqref{drei}.
\\[3mm] First note that the construction is in fact independent of the choice of a Hermitian metric $h^E$. Indeed, a variation of $h^E$ does not change the odd-signature operator $\B$ as a differential operator. However it enters a priori the definition of $\dom (\B)$, since $h^E$ defines the $L^2-$Hilbert space.
\\[3mm] Recall that different Hermitian metrics give rise to equivalent $L^2-$norms over compact manifolds. Hence a posteriori the domain $\dom (\B)$ is indeed independent of the particular choice of $h^E$. 
\\[3mm] Independence of the choice of a Hermitian metric $h^E$ is essential, since for non-unitary flat vector bundles there is no canonical choice of $h^E$ and Hermitian metric is fixed arbitrarily. 
\\[3mm] Consider a smooth family $g^M(t), t\in \R$ of Riemannian metrics on $M$. Denote by $\GG_t$ the corresponding chirality operator in the sence of Definition \ref{chirality} and denote the associated refined torsion (recall \eqref{finite-torsion}) of the complex $(\domm_{t,[0,\lambda]},\DD_{t,[0,\lambda]})$ by $\rho_{t,[0,\lambda]}$. 
\\[3mm] Let $\B(t)=\B(\D, g^M(t))$ be the odd-signature operator corresponding to the Riemannian metric $g^M(t)$. Fix $t_0\in \R$ and choose $\lambda \geq 0$ such that there are no eigenvalues of $\B(t_0)^2$ of absolute value $\lambda$. Then there exists $\delta>0$ small enough such that the same holds for the spectrum of $\B(t)^2$ for $|t-t_0|<\delta$. Under this setup we obtain:
\begin{prop}\label{anomaly1}
Let the family $g^M(t)$ vary only in a compact subset of the interior of $M$. Then
$\exp(\xi_{\lambda}(\D, g^M(t)))\cdot \rho_{t,[0,\lambda]}$ is independent of $t\in (t_0-\delta,t_0+\delta)$.
\end{prop}
\begin{proof}
The arguments of [BK2, Lemma 9.2] are of local nature and transfer ad verbatim to the present situation for metric variations in the interior of the manifold. Hence the assertion follows for Riemannian metric remaining fixed in an open neighborhood of the boundary.
\end{proof}
\begin{prop}\label{anomaly2}
Denote the trivial connection on the trivial line bundle $M\times \C$ by $\D_{\textup{trivial}}$. Consider the even part of the associated odd-signature operator (recall Definition \ref{odd-signature})
$$\B_{\textup{trivial}}=\B_{\textup{even}}(\D_{\textup{trivial}}).$$ Indicate the metric dependence by $\B_{\textup{trivial}}(t):=\B_{\textup{trivial}}(g^M)$. Then $$\eta(\B^{(\lambda,\infty)}_{\textup{even}}(t))-\textup{rank}(E)\eta(\B_{\textup{trivial}}(t))\quad \textup{mod} \, \Z$$ is independent of $t\in (t_0-\delta,t_0+\delta)$.
\end{prop}
\begin{proof}
Indicate the dependence of $\domm^*_{[0,\lambda]}$ on $g^M(t)$ by $$\domm^k_{[0,\lambda]}(t):=\textup{Image}\, \Pi_{\B(t)^2, [0,\lambda]}\cap \domm^k.$$
Note first the by the choice of $\delta >0$ $$\dim \domm^k_{[0,\lambda]}(t)=\textup{const}, \quad t\in (t_0-\delta,t_0+\delta).$$ Since $\B^{[0,\lambda]}_{\textup{even}}(t)$ is finite-dimensional, we infer from the definition of the eta-invariant (cf. [BK2, (9.11)])
\begin{align}\label{konstantin}
\eta (\B^{[0,\lambda]}_{\textup{even}}(t))\equiv \frac{1}{2}\dim \domm^k_{[0,\lambda]}(t) \equiv \textup{const} \ \textup{mod} \, \Z, \quad t\in (t_0-\delta,t_0+\delta).
\end{align}
By construction 
$$\eta(\B_{\textup{even}}(t))=\eta(\B^{(\lambda,\infty)}_{\textup{even}}(t))+\eta(\B^{[0,\lambda]}_{\textup{even}}(t)).$$ Hence, in view of \eqref{konstantin}, it suffices (modulo $\Z$) to study the metric dependence of the eta-invariant of $\eta(\B_{\textup{even}}(t))$. 
\\[3mm] View $\B_{\textup{even}}(t)$ as a pair of a differential operator $P_E(t)$ with its boundary conditions $Q_E(t)$. Similarly view $\B_{\textup{trivial}}(t)$ as a pair $(P_{\C}(t), Q_{\C}(t))$. Note that by construction the pair $(P_E(t),Q_E(t))$ is locally isomorphic to $(P_{\C}(t), Q_{\C}(t))\times \one^k$, since the flat connection $\D$ is locally trivial in appropriate local trivializations.
\\[3mm] Since the variation of the eta-invariants is computed from the local information of the symbols (cf. [GS1, Theorem 2.8, Lemma 2.9]), we find that the difference
\begin{align*}
\eta(\B_{\textup{even}}(t))-\textup{rank}(E)\eta(\B_{\textup{trivial}}(t))=\\=\eta(P_E(t),Q_E(t))-\textup{rank}(E)\eta(P_{\C}(t), Q_{\C}(t))
\end{align*}
is independent of $t\in \R$ modulo $\Z$. The modulo $\Z$ reduction is needed to annihilate discontinuity jumps arising from eigenvalues crossing the imaginary axis. This proves the statement of the proposition.
\end{proof} 
\begin{prop}\label{anomaly3}
Let $\B(\D_{\textup{trivial}})$ denote the odd-signature operator (Definition \ref{odd-signature}) associated to the trivial line bundle $M\times \C$ with the trivial connection $\D_{\textup{trivial}}$. Consider in correspondence to \eqref{xi2} the expression 
\begin{align*}
\xi'(\D_{\textup{trivial}}, g^M(t))=\frac{1}{2}\sum_{k=0}^m(-1)^{k}\cdot k\cdot \zeta_{2\theta}(s=0,(\B(\D_{\textup{trivial}}, g^M(t))^2\restriction \domm^k).
\end{align*}
Then $$\xi'_{\lambda}(\D,g^M(t))-\textup{rank}(E)\cdot \xi'(\D_{\textup{trivial}},g^M(t))\quad \textup{mod}\ \Z$$ is independent of $t\in \R$.
\end{prop}
\begin{proof}
We show first that modulo $\Z$ it suffices to study the metric dependence of 
\begin{align*}
\xi'(\D, g^M(t)):=\frac{1}{2}\sum_{k=0}^m(-1)^{k}\cdot k\cdot \zeta_{2\theta}(s=0,(\B(\D, g^M(t))^2\restriction \domm^k).
\end{align*}
Indeed, by construction we have
$$\xi'(\D, g^M(t))=\xi'_{\lambda}(\D, g^M(t))+\frac{1}{2}\sum_{k=0}^m(-1)^k\cdot k\cdot \dim \domm^k_{(0,\lambda]}(t).$$
Anticipating the auxiliary result of Lemma \ref{modulo-2z} (iii) below, we obtain $$\xi'(\D, g^M(t))\equiv \xi'_{\lambda}(\D, g^M(t))\quad \textup{mod}\ \Z.$$
Recall that $\B(\D_{\textup{trivial}},g^M)\times \one^{\textup{rk}E}$ and $\B(\D, g^M)$ are locally isomorphic, as already encountered in the proof of Proposition \ref{anomaly2}. Now the statement of the proposition follows from the fact that the value of a zeta function at zero is given, modulo $\Z$ in order to avoid $\dim \ker \B(t)\in \Z$, by integrands of local invariants of the operator and its boundary conditions. 
\end{proof}
\begin{lemma}\label{modulo-2z}
Let $\mathcal{I}\subset \R$ denote any bounded intervall. Then 
\begin{enumerate}
\item $\frac{1}{2}\sum_{k=0}^m(-1)^{k+1}\cdot k\cdot \dim \domm^k_{\mathcal{I}}\equiv \frac{\dim M}{2}\dim \domm^{\textup{even}}_{\mathcal{I}}\ \textup{mod}\ 2\Z.$
\item If $0\notin \mathcal{I}$, then $\dim \domm^{\textup{even}}_{\mathcal{I}}\equiv 0 \ \textup{mod}\ 2\Z,$ 
\item If $0\notin \mathcal{I}$, then $\frac{1}{2}\sum_{k=0}^m(-1)^{k+1}\cdot k\cdot \dim \domm^k_{\mathcal{I}}\equiv 0 \ \textup{mod}\ \Z.$
\end{enumerate}
\end{lemma}
\begin{proof}
Note first the following relation $$\B^2_k=\GG \circ \B^2_{m-k}\circ \GG.$$
Hence with $r=(m+1)/2$ we obtain:
\begin{align}\label{modulo-z}
\frac{1}{2}\sum_{k=0}^m(-1)^{k+1}\cdot k\cdot \dim \domm^k_{\mathcal{I}}=
\frac{1}{2}\sum_{k=0}^{r-1}(m-4k)\cdot \dim \domm^{2k}_{\mathcal{I}}=\\
=\frac{m}{2}\dim \domm^{\textup{even}}_{\mathcal{I}}-2\sum_{k=0}^{r-1}k \cdot \dim \domm^{2k}_{\mathcal{I}}.
\end{align}
This proves the first statement. For the second statement assume $0 \notin \mathcal{I}$ till the end of the proof. Consider the operators 
\begin{align}\label{stern1}
\B^{+,\mathcal{I}}_k=\GG_{\mathcal{I}}\DD_{\mathcal{I}}:\domm^k_{\mathcal{I}}\cap \ker (\DD_{\mathcal{I}}\GG_{\mathcal{I}})\rightarrow \domm^{m-k-1}_{\mathcal{I}}\cap \ker (\DD_{\mathcal{I}}\GG_{\mathcal{I}}), \\ \label{stern2}
\B^{-,\mathcal{I}}_k=\DD_{\mathcal{I}}\GG_{\mathcal{I}}:\domm^k_{\mathcal{I}}\cap \ker (\GG_{\mathcal{I}}\DD_{\mathcal{I}})\rightarrow \domm^{m-k+1}_{\mathcal{I}}\cap \ker (\GG_{\mathcal{I}}\DD_{\mathcal{I}}).
\end{align}
Since $0 \notin \mathcal{I}$, the maps $\B^{\pm, \mathcal{I}}_k$ are isomorphisms by bijectivity of the map in \eqref{bijective-two}. Furthermore they commute with $(\B^{\pm, \mathcal{I}})^2$ in the following way
\begin{align}\label{BB-BB}
\B^{\pm, \mathcal{I}}_k \circ [(\B^{\pm, \mathcal{I}})^2\restriction \domm^k]=[(\B^{\pm, \mathcal{I}})^2\restriction \domm^{m-k\mp 1}]\circ \B^{\pm, \mathcal{I}}_k.
\end{align}
Hence we obtain with $\domm^{\pm,k}_{\mathcal{I}}$ denoting the span of generalized eigenforms of $(\B^{\pm,\mathcal{I}})^2 \restriction \domm^k$ the following relations
\begin{align*}
\dim \domm^{+,k}_{\mathcal{I}}=\dim \domm^{+,m-k-1}_{\mathcal{I}}, \\
\dim \domm^{-,k}_{\mathcal{I}}=\dim \domm^{-,m-k+1}_{\mathcal{I}}.
\end{align*} 
Due to $\dim \domm^{\textup{even}}_{\mathcal{I}}=\dim \domm^{+,\textup{even}}_{\mathcal{I}}+\dim \domm^{-,\textup{even}}_{\mathcal{I}}$ this implies (recall $M$ is odd-dimensional)
\begin{align}\label{stern3}
\dim \domm^{\textup{even}}_{\mathcal{I}}\equiv \dim \domm^{+,2p}_{\mathcal{I}} \ \textup{mod} \ 2\Z, \textup{if} \ \dim M=4p+1, \\ \label{stern4}
\dim \domm^{\textup{even}}_{\mathcal{I}}\equiv \dim \domm^{-,2p}_{\mathcal{I}} \ \textup{mod} \ 2\Z, \textup{if} \ \dim M=4p-1.
\end{align}
Finally recall the explicit form of $(\B^{\pm})^2$:
\begin{align*}
(\B^+)^2=\left( \begin{array}{cc} \G \Da \G \Dr & 0 \\ 0 & \G \Dr \G \Da \end{array}\right)=:\left(\begin{array}{cc} D^+_1 & 0 \\ 0 & D^+_2 \end{array}\right), \\
(\B^-)^2=\left( \begin{array}{cc} \Dr \G \Da \G  & 0 \\ 0 & \Da \G \Dr \G \end{array}\right)=:\left(\begin{array}{cc} D^-_1 & 0 \\ 0 & D^-_2 \end{array}\right).
\end{align*}
Moreover we put $$(\B^{\pm, \mathcal{I}})^2\restriction \domm^k= D^{\pm , \mathcal{I}}_{1,k}\oplus D^{\pm , \mathcal{I}}_{2,k}.$$
Note the following relations
\begin{align*}
&(\G \Dr)\circ D^+_1=D^+_2\circ (\G \Dr), \\
&D^+_1\circ (\G \Da)=(\G \Da)\circ D^+_2; \\
&\hspace{30mm} (\Da \G)\circ D^-_1=D^-_2\circ (\Da \G), \\
&\hspace{30mm} D^-_1\circ (\Dr \G)= (\Dr \G)\circ D^-_2.
\end{align*}
Due to $0 \notin \mathcal{I}$ these relations imply, similarly to \eqref{BB-BB}, spectral equivalence of $D^{\pm, \mathcal{I}}_{1,k}$ and $D^{\pm, \mathcal{I}}_{2,k}$ in the middle degree $k=2p$ for $\dim M=4p \pm 1$, respectively. This finally yields the desired relations
\begin{align*}
\dim \domm^{\textup{even}}_{\mathcal{I}}\equiv \dim \domm^{+,2p}_{\mathcal{I}}\equiv 0 \ \textup{mod} \ 2\Z, \textup{if} \ \dim M=4p+1, \\ 
\dim \domm^{\textup{even}}_{\mathcal{I}}\equiv \dim \domm^{-,2p}_{\mathcal{I}}\equiv 0 \ \textup{mod} \ 2\Z, \textup{if} \ \dim M=4p-1.
\end{align*}
\end{proof}\ \\
\\[-7mm] Propositions \ref{anomaly1}, \ref{anomaly2} and \ref{anomaly3} determine together the metric anomaly of $\rho(\D, g^M, h^E)$ up to a sign and we deduce the following central corollary.
\begin{cor}\label{RAT-sign}
Let $M$ be an odd-dimensional oriented compact Riemannian manifold. Let $(E, \D, h^E)$ be a flat complex vector bundle over $M$. Denote by $\D_{\textup{trivial}}$ the trivial connection on $M\times \C$ and let $\B_{\textup{trivial}}$ denote the even part of the associated odd-signature operator. Then 
\begin{align*}
\rho_{\textup{an}}(\D):=\rho(\D, g^M, h^E)\cdot \exp\left[i\pi \, \textup{rk}(E)(\eta(\B_{\textup{trivial}}(g^M)) + \xi'(\D_{\textup{trivial}}, g^M))\right]
\end{align*}
is modulo sign independent of the choice of $g^M$ in the interior of $M$.
\end{cor}\ \\
\\[-7mm] In view of the corollary above we can now define the "refined analytic torsion". It will be a differential invariant in the sense, that even though defined by geometric data in form of the metric structures, it is shown to be independent of their form in the interior of the manifold.
\begin{defn}\label{rho-def}
Let $M$ be an odd-dimensional oriented Riemannian manifold. Let $(E,\D)$ be a flat complex vector bundle over $M$. Then the refined analytic torsion is defined as the equivalence class of $\rho_{\textup{an}}(\D)$ modulo multiplication by $\exp [i \pi]$: $$\rho_{\textup{an}}(M, E):=\rho_{\textup{an}}(\D) /_{e^{i\pi}}.$$
\end{defn} \ \\
\\[-7mm] Note that the sign indeterminacy is also present in the original construction by Braverman and Kappeler, see [BK2, Remark 9.9 and Remark 9.10]. In the presentation below, we refer to the representative $\rho_{\textup{an}}(\D)$ of the class $\rho_{\textup{an}}(M, E)$ as refined analytic torsion, as well.

\section{Ray-Singer norm of Refined analytic torsion}\label{RS} \
\\[-3mm] Recall first the construction of the Ray-Singer torsion as a norm on the determinant line bundle for compact oriented Riemannian manifolds. Let $(M,g^M)$ and $(E,\D,h^E)$ be as in Subsection \ref{explicit-unitary}. 
\\[3mm] Let $\triangle_{\textup{rel}}$ be the Laplacian associated to the Fredholm complex $(\domr,\Dr)$ defined at the beginning of Section \ref{explicit-unitary}. As in \eqref{decomposition} in case of the squared odd-signature operator $\B^2$, it induces a spectral decomposition into a direct sum of subcomplexes for any $\lambda \geq 0$. $$(\domr, \Dr)=(\domr^{[0,\lambda]}, \Dr^{[0,\lambda]})\oplus (\domr^{(\lambda, \infty)}, \Dr^{(\lambda, \infty)}).$$ The scalar product on $\domr^{[0,\lambda]}$ induced by $g^M$ and $h^E$, induces a norm on the determinant line $\textup{Det}(\domr^{[0,\lambda]}, \Dr^{[0,\lambda]})$ (we use the notation of determinant lines of finite dimensional complexes in [BK2, Section 1.1]). There is a canonical isomorphism $$\phi_{\lambda}:\textup{Det}(\domr^{[0,\lambda]}, \Dr^{[0,\lambda]})\to \textup{Det}H^*(\domr, \Dr),$$ induced by the Hodge-decomposition in finite-dimensional complexes. Choose on $\textup{Det}H^*(\domr, \Dr)$ the norm $\|\cdot\|^{\textup{rel}}_{\lambda}$ such that $\phi_{\lambda}$ becomes an isometry. Further denote by $T^{RS}_{(\lambda, \infty)}(\Dr)$ the scalar analytic torsion associated to the complex $(\domr^{(\lambda, \infty)}, \Dr^{(\lambda, \infty)})$:
$$T^{RS}_{(\lambda, \infty)}(\Dr):=\exp \left(\frac{1}{2}\sum_{k=1}^m(-1)^{k+1}\cdot k\cdot \zeta'(s=0, \triangle^{(\lambda, \infty)}_{k, \textup{rel}})\right),$$
where $\triangle^{(\lambda, \infty)}_{\textup{rel}}$ is the Laplacian associated to the complex $(\domr^{(\lambda, \infty)}, \Dr^{(\lambda, \infty)})$. Note the difference to the sign convention of [RS]. However we are consistent with [BK2].
\\[3mm]The Ray-Singer norm on $\textup{Det}H^*(\domr, \Dr)$ is then defined by 
\begin{align}\label{norm-rel}
\|\cdot\|^{RS}_{\textup{Det}H^*(\domr, \Dr)}:=\|\cdot\|^{\textup{rel}}_{\lambda}\cdot T^{RS}_{(\lambda, \infty)}(\Dr).
\end{align}
With a completely analogous construction we obtain the Ray-Singer norm on the determinant line $\textup{Det}H^*(\doma, \Da)$
\begin{align}\label{norm-abs}
\|\cdot\|^{RS}_{\textup{Det}H^*(\doma, \Da)}:=\|\cdot\|^{\textup{abs}}_{\lambda}\cdot T^{RS}_{(\lambda, \infty)}(\Da).
\end{align}
Both constructions turn out to be independent of the choice of $\lambda \geq 0$, which follows from arguments analogous to those in the proof of Proposition \ref{rho-element}. In fact we get for $0 \leq \lambda < \mu$: 
\begin{align*}
\|\cdot \|^{\textup{rel/abs}}_{\mu}=\|\cdot \|^{\textup{rel/abs}}_{\lambda} \cdot T^{RS}_{(\lambda, \mu ]}(\D_{\textup{min/max}}),
\end{align*}
which implies that the Ray-Singer norms are well-defined. Furthermore by the arguments in [Mu, Theorem 2.6] the norms do not depend on the metric structures in the interior of the manifold.
\begin{remark}
Note that the Ray-Singer analytic torsion considered in [V] and [L\"{u}] differs from our setup in the sign convention and by the absence of factor $1/2$.
\end{remark} \ \\
\\[-7mm] We can apply the same construction to the Laplacian of the complex $(\domm, \DD)$ introduced in Definition \ref{domain} $$(\domm, \DD)=(\domr, \Dr)\oplus (\doma, \Da).$$ Similarly we obtain
\begin{align}\label{norm}
\|\cdot\|^{RS}_{\textup{Det}H^*(\domm, \DD)}:=\|\cdot\|_{\lambda}\cdot T^{RS}_{(\lambda, \infty)}(\DD).
\end{align}
This "doubled" Ray-Singer norm is naturally related to the previous two norms in \eqref{norm-rel} and \eqref{norm-abs}. There is a canonical "fusion isomorphism", cf. [BK2, (2.18)] for general complexes of finite dimensional vector spaces
\begin{align}\nonumber
\mu: \textup{Det}H^*(\domr, \Dr)\oplus \textup{Det}H^*(\doma, \Da) \to \textup{Det}H^*(\domm, \DD), 
\\ \label{fusion}\textup{such that}\ \|\mu(h_1\otimes h_2)\|_{\lambda}=\|h_1\|^{\textup{rel}}_{\lambda}\cdot \|h_2\|^{\textup{abs}}_{\lambda},
\end{align}
where we recall $(\domm, \DD)=(\domr, \Dr)\oplus (\doma, \Da)$ by definition. Further we have by the definition of $(\domm, \DD)$ following relation between the scalar analytic torsions: 
\begin{align}\label{scalar}
T^{RS}_{(\lambda, \infty)}(\DD)=T^{RS}_{(\lambda, \infty)}(\Dr)\cdot T^{RS}_{(\lambda, \infty)}(\Da).
\end{align}
Combining \eqref{fusion} and \eqref{scalar} we end up with a relation between norms 
\begin{align}
\|\mu(h_1\otimes h_2)\|^{RS}_{\textup{Det}H^*(\domm, \DD)}=\|h_1\|^{RS}_{\textup{Det}H^*(\domr, \Dr)}\cdot \|h_2\|^{RS}_{\textup{Det}H^*(\doma, \Da)}. 
\end{align}
The next theorem provides a motivation for viewing $\rho_{\textup{an}}(\D)$ as a refinement of the Ray-Singer torsion.
\begin{thm}\label{rho-norm}
Let $M$ be a smooth compact odd-dimensional oriented Riemannian manifold. Let $(E,\D, h^E)$ be a flat complex vector bundle over $M$ with a flat Hermitian metric $h^E$. Then $$\|\rho_{\textup{an}}(\D)\|^{RS}_{\textup{Det}H^*(\domm, \DD)}=1.$$
\end{thm}
\begin{proof}
Recall from the assertion of Theorem \ref{log-det-gr} 
\begin{align*}
\det\nolimits_{gr} (\B^{(\lambda, \infty)}_{\textup{even}})=e^{\xi_{\lambda}(\D, g^M)}\cdot e^{-i\pi \xi'_{\lambda}(\D,g^M)}\cdot e^{-i\pi\eta(\B_{\textup{even}})},
\end{align*}
Flatness of $h^E$ implies by construction that $\B^2=\triangle_{\textup{rel}}\oplus \triangle_{\textup{abs}}$ and hence
\begin{align*}
\xi_{\lambda}(\D, g^M)=-\log T^{RS}_{(\lambda, \infty)}(\DD).
\end{align*}
Further $\B_{\textup{even}}$ is self-adjoint and thus has a real spectrum. Hence $\eta(\B_{\textup{even}})$ and $\xi'_{\lambda}(\D,g^M)$ are real-valued, as well. Thus we derive 
\begin{align}\label{vier}
\left|\det\nolimits_{gr} (\B^{(\lambda, \infty)}_{\textup{even}})\right|=\frac{1}{T^{RS}_{(\lambda, \infty)}(\DD)}.
\end{align}
Furthermore we know from [BK2, Lemma 4.5], which is a general result for complexes of finite-dimensional vector spaces,
\begin{align}\label{5}
\|\rho_{[0,\lambda]}\|_{\lambda}=1.
\end{align}
Now the assertion follows by combining the definition of the refined analytic torsion with \eqref{vier}, \eqref{5} and the fact that the additional terms annihilating the metric anomaly are all of norm one. In fact we have: 
\begin{align*}
\|\rho_{\textup{an}}(\D)\|^{RS}_{\textup{Det}H^*(\domm, \DD)}=\left|\det\nolimits_{gr} (\B^{(\lambda, \infty)}_{\textup{even}})\right| \cdot T^{RS}_{(\lambda, \infty)}(\DD) \cdot \|\rho_{[0,\lambda]}\|_{\lambda} = 1.
\end{align*}
\end{proof}\ \\
\\[-7mm] If the Hermitian metric is not flat, the situation becomes harder. In the setup of closed manifolds M. Braverman and T. Kappeler performed a deformation procedure in [BK2, Section 11] and proved in this way the relation between the Ray-Singer norm and the refined analytic torsion in  [BK2, Theorem 11.3]. 
\\[3mm] Unfortunately the deformation argument is not local and the arguments in [BK2] do not apply in the setup of manifolds with boundary. Nevertheless we can derive appropriate result by relating our discussion to the closed double manifold. 
\\[3mm] Assume the metric structures $(g^M, h^E)$ to be product near the boundary $\partial M$. The issues related to the product structures are discussed in detail in [BLZ, Section 2]. More precisely, we identify using the inward geodesic flow a collar neighborhood $U\subset M$ of the boundary $\partial M$ diffeomorphically with $[0,\epsilon)\times \partial M, \epsilon > 0$. Explicitly we have the diffeomorphism 
\begin{align*}
\phi^{-1}:[0,\epsilon)\times \partial M &\rightarrow U, \\
(t,p)& \mapsto \gamma_p(t),
\end{align*}
where $\gamma_p$ is the geodesic flow starting at $p \in \partial M$ and $\gamma_p(t)$ is the geodesics from $p$ of length $t \in [0,\epsilon)$. The metric $g^M$ is product near the boundary, if over $U$ it is given under the diffeomorphism $\phi: U \to [0,\epsilon)\times \partial M$ by 
\begin{align}
\phi_*g^M|_U=dx^2\oplus g^M|_{\partial M}.
\end{align}
The diffeomorphism $U \cong [0,\epsilon)\times \partial M$ shall be covered by a bundle isomorphism $\widetilde{\phi}: E|_U \to [0,\epsilon)\times E|_{\partial M}$. The fiber metric $h^E$ is product near the boundary, if it is preserved by the bundle isomorphism, i.e. 
\begin{align}
\widetilde{\phi}_*h^E|_{\{x\}\times \partial M}=h^E|_{\partial M}.
\end{align}
The assumption of product structures guarantees that the closed double manifold $$\mathbb{M}=M\cup_{\partial M}M$$
is a smooth closed Riemannian manifold and the Hermitian vector bundle $(E,h^E)$ extends to a smooth Hermitian vector bundle $(\mathbb{E},h^{\mathbb{E}})$ over the manifold $\mathbb{M}$. 
\\[3mm] Moreover we assume the flat connection $\D$ on $E$ to be in \emph{temporal gauge}. The precise definition of a connection in temporal gauge and the proof of the fact that each flat connection is gauge-equivalent to a flat connection in temporal gauge, are provided in [BV4, ]. 
\\[3mm] The assumption on $\D$ to be a flat connection in temporal gauge is required in the present context to guarantee that $\D$ extends to a smooth flat connection $\mathbb{D}$ on $\mathbb{E}$, with $$\mathbb{D}|_{M}=\D.$$
\begin{thm}\label{double}
Let $(M^m,g^M)$ be an odd-dimensional oriented and compact smooth Riemannian manifold with boundary $\partial M$. Let $(E,\D,h^E)$ be a flat Hermitian vector bundle with the Hermitian metric $h^E$, not necessarily flat. 
\\[3mm] Assume the metric structures $(g^M,h^E)$ to be product and the flat connection $\D$ to be in temporal gauge near the boundary $\partial M$. Then
$$ \|\rho_{\textup{an}}(\D)\|^{RS}_{\det H^*(\domm,\DD)}=\textup{exp}[\pi \textup{Im}\, \eta (\B_{\textup{even}}(g^M))].$$
\end{thm}
\begin{proof}
By assumption we obtain a closed Riemannian double manifold $(\mathbb{M},g^{\mathbb{M}})$ and a flat Hermitian vector bundle $(\mathbb{E}, \mathbb{D}, h^{\mathbb{E}})$ over $\mathbb{M}$ with a flat Hermitian metric $h^{\mathbb{E}}$. Denote by $(\dom, \mathbb{D})$ the unique boundary conditions (see [BL1]) of the twisted de Rham complex $(\Omega^*(\mathbb{M}, \mathbb{E}),\mathbb{D})$. Denote the closure of $\Omega^*(\mathbb{M}, \mathbb{E})$ with respect to the $L^2-$scalar product defined by $g^{\mathbb{M}}$ and $h^{\mathbb{E}}$, by $L^2_*(\mathbb{M}, \mathbb{E})$.
\\[3mm] The Riemannian metric $g^{\mathbb{M}}$ gives rise to the Hodge star operator $*$ and we set $$\mathbb{G}:=i^r(-1)^{\frac{k(k+1)}{2}}*:\Omega^k(\mathbb{M}, \mathbb{E})\rightarrow \Omega^{k-1}(\mathbb{M}, \mathbb{E}), \quad r:=(\dim M +1)/2$$
which extends to a self-adjoint involution on $L^2_*(\mathbb{M}, \mathbb{E})$. We define the odd signature operator $\mathbb{B}$ of the Hilbert complex $(\dom, \mathbb{D})$: $$\mathbb{B}:=\mathbb{G}\mathbb{D}+\mathbb{D}\mathbb{G}.$$
This is precisely the odd-signature operator associated to the closed manifold $\mathbb{M}$, as used in the construction of [BK1, BK2].
\\[3mm] Note that we now have two triples: the triple $(\mathbb{D}, \mathbb{G}, \mathbb{B})$ associated to the closed manifold $\mathbb{M}$ and the triple $(\DD, \GG, \B)$ associated to $(M, \partial M)$, as defined in Subsection \ref{explicit-unitary}. 
\\[3mm] Consider now the diffeomorphic involution on the closed double $$\A: \mathbb{M}\rightarrow \mathbb{M},$$ interchanging the two copies of $M$. It gives rise to an isomorphism of Hilbert complexes $$\A^*: (\dom, \mathbb{D})\rightarrow (\dom, \mathbb{D}),$$
which is an involution as well. We get a decomposition of $(\dom , \mathbb{D})$ into the $(\pm 1)$-eigenspaces of $\A^*$, which form subcomplexes of the total complex:
\begin{align}\label{involution-decomp}
(\dom , \mathbb{D})=(\dom^+ , \mathbb{D}^+)\oplus (\dom^- , \mathbb{D}^-),
\end{align}
where the upper-indices $\pm$ refer to the $(\pm 1)$-eigenspaces of $\A^*$, respectively.
\\[3mm] The central property of the decomposition, by similar arguments as in [BL1, Theorem 4.1], lies in the following observation
\begin{align}
\dom^+|_M=\doma, \quad \dom^-|_M=\domr.
\end{align}
By the symmetry of the elements in $\dom^{\pm}$ we obtain the following natural isomorphism of complexes:
\begin{align*}
\Phi:(\dom , \mathbb{D})=(\dom^+ , \mathbb{D}^+)\oplus (\dom^- , \mathbb{D}^-)&\rightarrow (\dom_{\textup{max}}, \Da)\oplus (\dom_{\textup{min}}, \Dr), \\
\w=\w^+\oplus\w^-&\mapsto 2\w^+|_M \oplus 2\w^-|_M,
\end{align*}
which extends to an isometry with respect to the natural $L^2-$structures. Using the relations 
\begin{align}\label{G-double}
\Phi\circ \mathbb{D}\circ \Phi^{-1}=\DD, \quad
\Phi\circ \mathbb{G}\circ \Phi^{-1}=\GG, 
\end{align}
we obtain with $\Delta$ and $\widetilde{\triangle}$, denoting respectively the Laplacians of the complexes $(\dom, \mathbb{D})$ and $(\domm, \DD)\equiv (\domr, \Dr)\oplus (\doma, \Da)$:
\begin{align*}
\Phi \dom (\mathbb{B})=\dom (\B), \quad \Phi \circ \mathbb{B} \circ \Phi^{-1}=\B, \\
\Phi \dom (\Delta)=\dom (\widetilde{\triangle}), \quad \Phi \circ \Delta \circ \Phi^{-1}=\widetilde{\triangle}.
\end{align*}
Hence the odd-signature operators $\mathbb{B}, \B$ as well as the Laplacians $\Delta, \widetilde{\triangle}$ are spectrally equivalent. Consider the spectral projections $\Pi_{\mathbb{B}^2,[0,\lambda]}$ and $\Pi_{\B^2,[0,\lambda]}, \lambda \geq 0$ of $\mathbb{B}$ and $\B$ respectively, associated to eigenvalues of absolute value in $[0,\lambda]$. By the spectral equivalence $\mathbb{B}$ and $\B$ we find $$\Phi \circ \Pi_{\mathbb{B}^2,[0,\lambda]}=\Pi_{\B^2,[0,\lambda]}\circ \Phi.$$
Hence the isomorphism $\Phi$ reduces to an isomorphism of finite-dimensional complexes:
\begin{align*}
\Phi_{\lambda}:(&\dom_{[0,\lambda]}, \mathbb{D}_{[0,\lambda]})\xrightarrow{\sim} (\domm_{[0,\lambda]}, \DD_{[0,\lambda]}), \\
\textup{where} \quad &\dom_{[0,\lambda]}:=\dom \cap \textup{Image}\Pi_{\mathbb{B}^2,[0,\lambda]}, \\
&\domm_{[0,\lambda]}:=\domm \cap \textup{Image}\Pi_{\B^2,[0,\lambda]}.
\end{align*}
Moreover $\Phi_{\lambda}$ induces an isometric identification of the corresponding determinant lines, which we denote again by $\Phi_{\lambda}$, by a minor abuse of notation
$$\Phi_{\lambda}:\det (\dom_{[0,\lambda]}, \mathbb{D}_{[0,\lambda]})\xrightarrow{\sim} \det (\domm_{[0,\lambda]}, \DD_{[0,\lambda]}),$$ where we use the notation for determinant lines of finite-dimensional complexes in [BK2, Section 1.1]. By Corollary \ref{cohomology} we have the canonical identifications of determinant lines
\begin{align}\label{di1}
\det (\dom_{[0,\lambda]}, \mathbb{D}_{[0,\lambda]})\cong &\det H^*(\dom , \mathbb{D}), \\ \label{di2}
\det (\domm_{[0,\lambda]}, \DD_{[0,\lambda]})\cong &\det H^*(\domm, \DD),
\end{align}
The determinant lines on the left hand side of both identifications carry the natural $L^2-$Hilbert structure. Denote the norms on $\det H^*(\dom , \mathbb{D})$ and $\det H^*(\domm, \DD)$ which turn both identifications into isometries, by $\|\cdot \|_{\lambda}$ and $\|\cdot \|_{\lambda}^{\sim}$, respectively. Then we can view $\Phi_{\lambda}$ as 
$$\Phi_{\lambda}:\det H^*(\dom , \mathbb{D})\xrightarrow{\sim} \det H^*(\domm, \DD),$$
isometric with respect to the Hilbert structures induced by $\|\cdot \|_{\lambda}$ and $\|\cdot \|_{\lambda}^{\sim}$. 
\\[3mm] Finally, consider the refined torsion elements (not the refined analytic torsion) of the determinant lines, as defined in [BK2, Section 1.1], see also \eqref{finite-torsion} 
\begin{align*}
\rho^{\mathbb{G}}_{[0,\lambda]}\in \det (\dom_{[0,\lambda]}, \mathbb{D}_{[0,\lambda]})\cong \det H^*(\dom, \mathbb{D}), \\
\rho^{\GG}_{[0,\lambda]}\in \det (\domm_{[0,\lambda]}, \DD_{[0,\lambda]}) \cong \det H^*(\domm, \DD).
\end{align*}
We infer from \eqref{G-double} the following relation:
\begin{align*}
\Phi_{\lambda}\left( \rho^{\mathbb{G}}_{[0,\lambda]} \right) = \rho^{\GG}_{[0,\lambda]}, \quad \textup{hence:} \  \| \rho^{\mathbb{G}}_{[0,\lambda]} \|_{\lambda} = \| \rho^{\GG}_{[0,\lambda]} \|_{\lambda}^{\sim}.
\end{align*}
Together with spectral equivalence of $\Delta$ and $\widetilde{\triangle}$, as well as of $\mathbb{B}$ and $\B$, with similar statements for constructions on trivial line bundles $M\times \C$ and $\mathbb{M}\times \C$, we finally obtain
\begin{align}
\|\rho_{\textup{an}}(\mathbb{D})\|^{RS}_{\det H^*(\dom, \mathbb{D})}= \|\rho_{\textup{an}}(\D)\|^{RS}_{\det H^*(\domm, \DD)},
\end{align}
where $\rho_{\textup{an}}(\mathbb{D})$ denotes the refined analytic torsion as defined by M. Braverman and T. Kappeler in [BK2] and $\rho_{\textup{an}}(\D)$ denotes the refined analytic torsion in the sense of the present discussion.
\\[3mm] The statement now follows from [BK2, Theorem 11.3].
\end{proof}\ \\
\\[-7mm] In the setup of the previous theorem we can improve the sign indeterminacy of $\rho_{\textup{an}}(\D)$ as follows:
\begin{prop}\label{RAT-sign2}
Let $M$ be an odd-dimensional oriented compact Riemannian manifold. Let $(E, \D, h^E)$ be a flat complex vector bundle over $M$. Denote by $\D_{\textup{trivial}}$ the trivial connection on $M\times \C$ and let $\B_{\textup{trivial}}$ denote the even part of the associated odd-signature operator. 
\\[3mm] Assume the metric structures $(g^M,h^E)$ to be product and the flat connection $\D$ to be in temporal gauge near the boundary $\partial M$. Then 
\begin{align*}
\rho_{\textup{an}}(\D)=\rho(\D, g^M, h^E)\cdot \exp\left[i\pi \, \textup{rk}(E)(\eta(\B_{\textup{trivial}}(g^M)) + \xi'(\D_{\textup{trivial}}, g^M))\right]
\end{align*}
is independent of the choice of $g^M$ in the interior of $M$, up to multiplication by $$\exp [i \pi \textup{rank}(E)].$$ In particular it is independent of $g^M$ in the interior of $M$ for $E$ being a complex vector bundle of even rank.
\end{prop}
\begin{proof}
Consider a smooth family $g^M(t),t\in \R$ of Riemannian metrics, variing only in the interior of $M$ and being of fixed product structure near $\partial M$. By arguments in Theorem \ref{double} we can relate $\B(g^M(t))$ to operators on the closed double $\mathbb{M}$ and deduce from [BK1, Theorem 5.7] that $\rho(\D, g^M(t),h^E)$ is continuous in $t$. However 
\begin{align*}
\exp\left[i\pi \, \textup{rk}(E)\eta(\B_{\textup{trivial}}(g^M(t)))\right]
\end{align*}
is continuous in $t\in \R$ only up to multiplication by $e^{i\pi \textup{rk}E}$. Hence the element $\rho_{\textup{an}}(\D)$, where we denote the a priori metric dependence by $\rho_{\textup{an}}(\D, g^M(t))$, is continuous in $t$ only modulo multiplication by $e^{i \pi \textup{rk}(E)}$. For $g^M(t)$ varying only in the interior of $M$ and any $t_0, t_1 \in \R$ we infer from the mod $\Z$ metric anomaly considerations in Propositions \ref{anomaly2} and \ref{anomaly3}: $$\rho_{\textup{an}}(\D, g^M(t_0))=\pm \rho_{\textup{an}}(\D, g^M(t_1)).$$ For rk$(E)$ odd this is already the desired statement, since $\exp (i\pi \textup{rk}(E))=-1$. For rk$(E)$ even, $\rho_{\textup{an}}(\D, g^M(t))$ is continuous in $t$ and nowhere vanishing, so the sign in the last relation must be positive. This proves the statement.
\end{proof}\ \\
\\[-7mm] In view of the corollary above we can re-define the refined analytic torsion in the setup of product metric structures and flat connection in temporal gauge as follows:
\begin{align}\label{RAT-sign3}
\rho_{\textup{an}}(M, E):=\rho_{\textup{an}}(\D) /_{e^{i\pi \textup{rank}(E)}}.
\end{align}

\begin{remark}
The interdeterminacy of $\rho_{\textup{an}}(\D)$ modulo multiplication by the factor $e^{i\pi \textup{rk}E}$ in fact corresponds and is even finer than the general indeterminacy in the construction of M. Braverman and T. Kappeler on closed manifolds, see [BK2, Remark 9.9 and Remark 9.10].
\end{remark}

\section{Open Problems}\label{open-refined} \
\\[-3mm] \emph{Ideal Boundary Conditions}
\\[3mm] As explained in the introduction, the approach of Braverman and Kappeler in [BK1, BK2] requires ideal boundary conditions for the twisted de Rham complex, which turn it into a Fredholm complex with Poincare duality and further provide elliptic boundary conditions for the associated odd-signature operator, viewed as a map between the even forms. In our construction we pursued a different strategy, however the question about existence of such boundary conditions remains. 
\\[3mm] This question was partly discussed in [BL1]. In view of [BL1, Lemma 4.3] it is not even clear whether ideal boundary conditions exist, satisfying Poincare duality and providing a Fredholm complex. For the approach of Braverman and Kappeler we need even more: the ideal boundary conditions need to provide elliptic boundary conditions for the odd-signature operator. We arrive at the natural open question, whether such boundary conditions exist.
\\[4mm] \emph{Conical Singularities} 
\\[3mm] Another possible direction for the discussion of refined analytic torsion is the setup of compact manifolds with conical singularities. At the conical singularity the question of appropriate boundary conditions is discussed in [Ch2], as well as in [BL2].
\\[3mm] It turns out that on odd-dimensional manifolds with conical singularities the topological obstruction is given by $H^{\nu}(N)$, where $N$ is the base of the cone and $\nu=\dim N /2$. If $$H^{\nu}(N)=0$$ then all ideal boundary conditions coincide and the construction of Braverman and Kappeler [BK1, BK2] goes through. Otherwise, see [Ch2, p.580] for the choice of ideal boundary conditions satisfying Poincare duality.
\\[4mm] \emph{Combinatorial Counterpart}
\\[3mm] Let us recall that the definition of the refined analytic torsion in [BK1, BK2] was partly motivated by providing analytic counterpart of the refined combinatorial torsion, introduced by V. Turaev in [Tu1].
\\[3mm] In his work V. Turaev introduced the notion of Euler structures and showed how it is applied to refine the concept of Reidemeister torsion by removing the ambiguities in choosing bases needed for construction. Moreover, Turaev observed in [Tu2] that on three-manifolds a choice of an Euler structure is equivalent to a choice of a Spin$^c$-structure. 
\\[3mm] Both, the Turaev-torsion and the Braverman-Kappeler refined torsion are holomorphic functions on the space of representations of the fundamental group on $GL(n,\C)$, which is a finite-dimensional algebraic variety. Using methods of complex analysis, Braverman and Kappeler computed the quotient between their and Turaev's construction.
\\[3mm] A natural question is whether this procedure has an appropriate equivalent for our proposed refined analytic torsion on manifolds with boundary. In our view this question can be answered affirmatively. 
\\[3mm] Indeed, by similar arguments as in [BK1, BK2] the proposed refined analytic torsion on manifolds with boundary can also be viewed as an analytic function on the finite-dimensional variety of representations of the fundamental group.
\\[3mm] For the combinatorial counterpart note that M. Farber introduced in [Fa] the concept of Poincare-Reidemeister metric, where using Poincare-duality in the similar spirit as in our construction, he constructed an invariantly defined Reidemeister torsion norm for non-unimodular representations. Further M. Farber and V. Turaev elaborated jointly in [FaTu] the relation between their concepts and introduced the refinement of the Poincare-Reidemeister scalar product. 
\\[3mm] The construction in [Fa] extends naturally to manifolds with boundary by similar means as in our definition of refined analytic torsion. This provides a combinatorial torsion norm on compact manifolds, well-defined without unimodularity assumption. It can then be refined in the spirit of [FaTu]. This would naturally provide the combinatorial counterpart for the presented refined analytic torsion. 

\section{References}\
\\[-1mm] [Ag] S. Agmon \emph{"On the eigenfunctions and on the eigenvalues of general elliptic boundary value problems"} Comm. Pure Appl. Math., vol. 15, 119-147 (1962)
\\[3mm] [APS] M. F. Atiyah, V.K. Patodi, I.M. Singer \emph{"Spectral asymmetry and Riemannian geometry I"}, Math. Proc. Camb. Phil. Soc. 77, 43-69 (1975)
\\[3mm] [BFK] D. Burghelea, L. Friedlander, T. Kappeler \emph{"Mayer-Vietoris type formula for determinants of elliptic differential operators"}, Journal of Funct. Anal. 107, 34-65 (1992)
\\[3mm] [BGV] N. Berline, E. Getzler, M. Vergne \emph{"Heat kernels and Dirac operators"}, Springer-Verlag, New Jork (1992)
\\[3mm] [BK1] M. Braverman and T. Kappeler \emph{"Refined analytic torsion"}, arXiv:math.DG/ 0505537v2, to appear in J. of. Diff. Geom.
\\[3mm] [BK2] M. Braverman and T. Kappeler \emph{"Refined Analytic Torsion as an Element of the Determinant Line"} , arXiv:math.GT/0510532v4, To appear in Geometry \& Topology 
\\[3mm] [BL1] J. Br\"{u}ning, M. Lesch \emph{"Hilbert complexes"}, J. Funct. Anal. 108, 88-132 (1992)
\\[3mm] [BL3] J. Br\"{u}ning, M. Lesch \emph{"On boundary value problems for Dirac type operators. I. Regularity and self-adjointness"}, arXiv:math/9905181v2 [math.FA] (1999)
\\[3mm] [BLZ] B. Booss, M. Lesch, C. Zhu \emph{"The Calderon Projection: New Definition and Applications"}, arXiv:math.DG/0803.4160v1 (2008)
\\[3mm] [BS] J. Br\"{u}ning, R. Seeley \emph{"An index theorem for first order regular singular operators"}, Amer. J. Math 110, 659-714, (1988)
\\[3mm] [BW] B. Booss, K. Wojchiechovski \emph{"Elliptic boundary problemsfor Dirac Operators"}, Birkh\"{a}user, Basel (1993) 
\\[3mm] [BV4] B. Vertman \emph{"Gluing Formula for Refined Analytic Torsion"}, preprint, arXiv:0808.0451 (2008)
\\[3mm] [BZ] J. -M. Bismut and W. Zhang \emph{"Milnor and Ray-Singer metrics on the equivariant determinant of a flat vector bundle"}, Geom. and Funct. Analysis 4, No.2, 136-212 (1994)
\\[3mm] [BZ1] J. -M. Bismut and W. Zhang \emph{"An extension of a Theorem by Cheeger and M\"{u}ller"}, Asterisque, 205, SMF, Paris (1992)
\\[3mm] [Ch] J. Cheeger \emph{"Analytic Torsion and Reidemeister Torsion"}, Proc. Nat. Acad. Sci. USA 74 (1977), 2651-2654
\\[3mm] [Fa] M. Farber \emph{Combinatorial invariants computing the Ray-Singer analytic torsion}, arXiv:dg-ga/ 9606014v1 (1996)
\\[3mm] [Gi] P.B. Gilkey \emph{"Invariance Theory, the Heat-equation and the Atiyah-Singer Index Theorem"}, Second Edition, CRC Press (1995)
\\[3mm] [Gi2] P.B. Gilkey \emph{"The eta-invariant and secondary characteristic classes of locally flat bundles"}, Algebraic and Differential topology $-$ global differential geometry, Teubner-Texte zur Math., vol. 70, Teubner, Leipzig, 49-87 (1984)
\\[3mm] [GS1] P. Gilkey, L. Smith \emph{"The eta-invariant for a class of elliptic boundary value problems"}, Comm. Pure Appl. Math. Vol.36, 85-131 (1983)
\\[3mm] [GS2] P. Gilkey, L. Smith \emph{"The twisted index problem for manifolds with boundary"}, J. Diff. Geom. 18, 393-444 (1983)
\\[3mm] [K] T. Kato \emph{"Perturbation Theory for Linear Operators"}, Die Grundlehren der math. Wiss. Volume 132, Springer (1966)
\\[3mm] [KL] P. Kirk and M. Lesch \emph{"The $\eta$-invariant, Maslov index and spectral flow for Dirac type operators on manifolds with boundary"}, Forum Math. 16, 553-629 (2004)
\\[3mm] [KM] F.F: Knudsen, D. Mumford \emph{"The projectivity of the moduli space of stable curves. I. Preliminaries on 'det' and 'Div'"}, Math. Scand. 39, no1, 19-55 (1976)
\\[3mm] [KN] S. Kobayashi, K. Nomizu \emph{"Foundations of differential geometry"}, Volume I, Interscience Publishers (1963)
\\[3mm] [L2] M. Lesch \emph{"Gluing formula in cohomological algebra"}, unpublished notes.
\\[3mm] [Lee] Y. Lee \emph{"Burghelea-Friedlander-Kappeler's gluing formula for the zeta-determinant and its application to the adiabatic decomposition of the zeta-determinant and the analytic torsion"}, Trans. Amer. Math. Soc., Vol. 355, 10, 4093-4110 (2003)
\\[3mm] [LR] J. Lott, M. Rothenberg \emph{"Analytic torsion for group actions"} J. Diff. Geom. 34, 431-481 (1991)
\\[3mm] [L\"{u}] W. L\"{u}ck \emph{"Analytic and topological torsion for manifolds with boundary and symmetry"}, J. Diff. Geom. 37, 263-322, (1993) 
\\[3mm] [Mi] J. Milnor \emph{"Whitehead torsion"}, Bull. Ams. 72, 358-426 (1966)
\\[3mm] [Mu] W. M\"{u}ller \emph{"Analytic torsion and R-torsion for unimodular representations"} J. Amer. Math. Soc., Volume 6, Number 3, 721-753 (1993) 
\\[3mm] [Mu1] W. M\"{u}ller \emph{"Analytic Torsion and R-Torsion of Riemannian manifolds"} Adv. Math. 28, 233-305 (1978)
\\[3mm] [Mun] J. Mukres \emph{"Elementary differential topology"} Ann. of Math. Stud. vol. 54, Princeton Univ. Press, Princeton, NJ (1961)
\\[3mm] [MZ1] X. Ma, W. Zhang \emph{"$\eta -$invariant and flat vector bundles I"}, Chinese Ann. Math. 27B, 67-72 (2006)
\\[3mm] [MZ2] X. Ma, W. Zhang \emph{"$\eta -$invariant and flat vector bundles II"}, Nankai Tracts in Mathematics. Vol. 11. World Scientific, 335-350, (2006)
\\[3mm] [Nic] L.I. Nicolaescu \emph{"The Reidemeister torsion of 3-manifolds"}, de Gruyter Studies in Mathematics, vol. 30, Berlin (2003)
\\[3mm] [Re1] K. Reidemeister \emph{"Die Klassifikation der Linsenr\"{a}ume"}, Abhandl. Math. Sem. Hamburg 11, 102-109 (1935)
\\[3mm] [Re2] K. Reidemeister \emph{"\"{U}berdeckungen von Komplexen"}, J. reine angew. Math. 173, 164-173 (1935)
\\[3mm] [ReS] M. Reed, B. Simon \emph{"Methods of Mathematical Physics"}, Vol. II, Acad.N.J. (1979)
\\[3mm] [Rh] G. de Rham \emph{"Complexes a automorphismes et homeomorphie differentiable"}, Ann. Inst. Fourier 2, 51-67 (1950)
\\[3mm] [RS] D.B. Ray and I.M. Singer \emph{"R-Torsion and the Laplacian on Riemannian manifolds"}, Adv. Math. 7, 145-210 (1971)
\\[3mm] [Ru] W. Rudin \emph{"Functional Analysis"}, Second Edition, Mc. Graw-Hill, Inc. Intern. Series in pure and appl. math. (1991)
\\[3mm] [RH] Rung-Tzung Huang \emph{"Refined Analytic Torsion: Comparison theorems and examples"}, math.DG/0602231v2 
\\[3mm] [Se1] R. Seeley \emph{"The resolvent of an elliptic boundary problem"} Amer. J. Math. 91 889-920 (1969)
\\[3mm] [Se2] R. Seeley \emph{"An extension of the trace associated with elliptic boundary problem"} Amer. J. Math. 91 963-983 (1969)
\\[3mm] [Sh] M.A. Shubin \emph{"Pseudodifferential operators and Spectral Theory"}, English translation: Springer, Berlin (1086) 
\\[3mm] [Tu1] V. G. Turaev \emph{"Euler structures, non-singular vector-fields and torsion of Reidemeister type"}, English Translation: Math. USSR Izvestia 34:3 627-662 (1990)
\\[3mm] [Tu2] V. G. Turaev \emph{"Torsion invariants of Spin$^c$ structures on three-manifolds"}, Math. Research Letters 4:5 679-695 (1997)
\\[3mm] [V] S. Vishik \emph{"Generalized Ray-Singer Conjecture I. A manifold with smooth boundary"}, Comm. Math. Phys. 167, 1-102 (1995)
\\[3mm] [Wh] J. H. Whitehead \emph{"Simple homotopy types"}, Amer. J. Math. 72, 1-57 (1950)

\end{document}